\documentclass{amsart}
\usepackage[all,knot]{xy}

\xyoption{arc} 
\usepackage{amsfonts}
\usepackage{array}
\usepackage{euscript}
\usepackage{tikz}
\usepackage{epsfig, psfrag, epic}
\usepackage{graphicx}
\usepackage{xy}
\usepackage[all]{xypic}      

\usepackage{pdfpages}

\usepackage{multirow}
\usepackage{amsmath}
\usepackage{enumerate}
\usepackage{color}
\usepackage{mathtools}
\usepackage{xcolor}
\usepackage[colorlinks=true,allbordercolors=white, citecolor=blue]{hyperref}

\usepackage[utf8]{inputenc}
\usepackage{subcaption}
\newcommand{\linep}[2]{ \ensuremath{  %
 \begin{xy}                           
  (0,-2)*+UR{\scriptstyle #1}="a",    
  (3,3)*+UR{\scriptstyle #2}="b",     
  "a";"b"**\dir{-}?>*\dir{>},         
 \end{xy}                             %
} }

\newcommand{\graphpp}[3]{ \ensuremath{
 \begin{xy}                           
  (0,-2)*+UR{\scriptstyle #1}="a",    
  (3,3)*+UR{\scriptstyle #2}="b",     
  (6,-2)*+UR{\scriptstyle #3}="c",    
  "a";"b"**\dir{-}?>*\dir{>},         %
  "b";"c"**\dir{-}?>*\dir{>}          %
 \end{xy}                             %
} }

\newcommand{\graphpm}[3]{ \ensuremath{
 \begin{xy}                           
  (0,-2)*+UR{\scriptstyle #1}="a",    
  (3,3)*+UR{\scriptstyle #2}="b",     
  (6,-2)*+UR{\scriptstyle #3}="c",    
  "a";"b"**\dir{-}?>*\dir{>},         %
  "c";"b"**\dir{-}?>*\dir{>}          %
 \end{xy}                             %
} }

\newcommand{\longgraph}[4]{ \ensuremath{
 \begin{xy}                           
  (0,-2)*+UR{\scriptstyle #1}="a",    
  (4,3)*+UR{\scriptstyle #2}="b",     
  (8,-2)*+UR{\scriptstyle #3}="c",    
  (12,3)*+UR{\scriptstyle #4}="d",    %
  "a";"b"**\dir{-}?>*\dir{>},         %
  "b";"c"**\dir{-}?>*\dir{>},         %
  "c";"d"**\dir{-}?>*\dir{>}          %
 \end{xy}                             %
} }

\input colordvi   

\title{Linking of letters and the lower central series of free groups}

\newtheorem{theorem}{Theorem}[section]
\newtheorem{corollary}[theorem]{Corollary}
\newtheorem{lemma}[theorem]{Lemma}
\newtheorem{proposition}[theorem]{Proposition}

\theoremstyle{definition}
\newtheorem{definition}[theorem]{Definition}

\newtheorem{example}[theorem]{Example}

\definecolor{darkgreen}{rgb}{0,0.5,0}

\newcommand{\sign}{\text{sign}}
\newcommand{\sym}{\mathcal{S}}

\newcommand{\R}{{\mathbb R}}
\newcommand{\Z}{{\mathbb Z}}
\newcommand{\Q}{{\mathbb Q}}

\newcommand{\E}{{\mathbb E}}
\newcommand{\fL}{{\mathbb L}}

\newcommand{\si}{{\mathcal S}}

\newcommand{\gs}{{\mathcal SG}}
\newcommand{\sy}{{\mathcal Symb}}

\newcommand{\Lie}{\mathcal{L}\!{\mathit ie\/}}    
\newcommand{\Eil}{\mathcal{E}\!{\mathit il\/}}    

\newcommand{\tr}{\mathcal{T}\!{\mathit r\/}}
\newcommand{\gr}{\mathcal{G}\!{\mathit r\/}}      

\newcommand\Item[1][]{%
  \ifx\relax#1\relax  \item \else \item[#1] \fi
  \abovedisplayskip=0pt\abovedisplayshortskip=0pt~\vspace*{-\baselineskip}}

\begin{document}

\author[J. Monroe]{Jeff Monroe}
\address{Mathematics Department,  University of Oregon}
\email{jmonroe@uoregon.edu}
\author[D. Sinha]{Dev Sinha}
\address{Mathematics Department,  University of Oregon}
\email{dps@uoregon.edu}

\begin{abstract}
We develop invariants of the lower central series of free groups through linking of letters,
showing they span 
the rational linear dual of the lower central series subquotients.
We build on an approach to Lie coalgebras through operads, setting the stage for generalization to the lower central series Lie algebra
of any group.  Our approach yields a new co-basis for free Lie algebras.  
We compare with the classical approach through Fox derivatives.  

\end{abstract}

\maketitle


\section{Introduction}

Linking is the derived version of intersection, defined through cobounding and then intersecting.  The simplest form
of linking is defined whenever two manifolds which cobound have total dimension one less than the ambient dimension.
Figure~\ref{introfigure} on the right shows a classical picture of linking, of three one-dimensional manifolds, labeled $a$, $b$ and $c$, in $\R^3$.
While this is the first example of linking in the typical order of  learning topology,
 the first example when ordering by dimension is that of zero-manifolds inside a one-manifold,
as illustrated on the left of Figure~\ref{introfigure}. 
(The second example would be a zero-manifold and a one-manifold in a two-manifold, so the usual linking would be the third case.)

\begin{figure}[h]
\includegraphics[width=5cm]{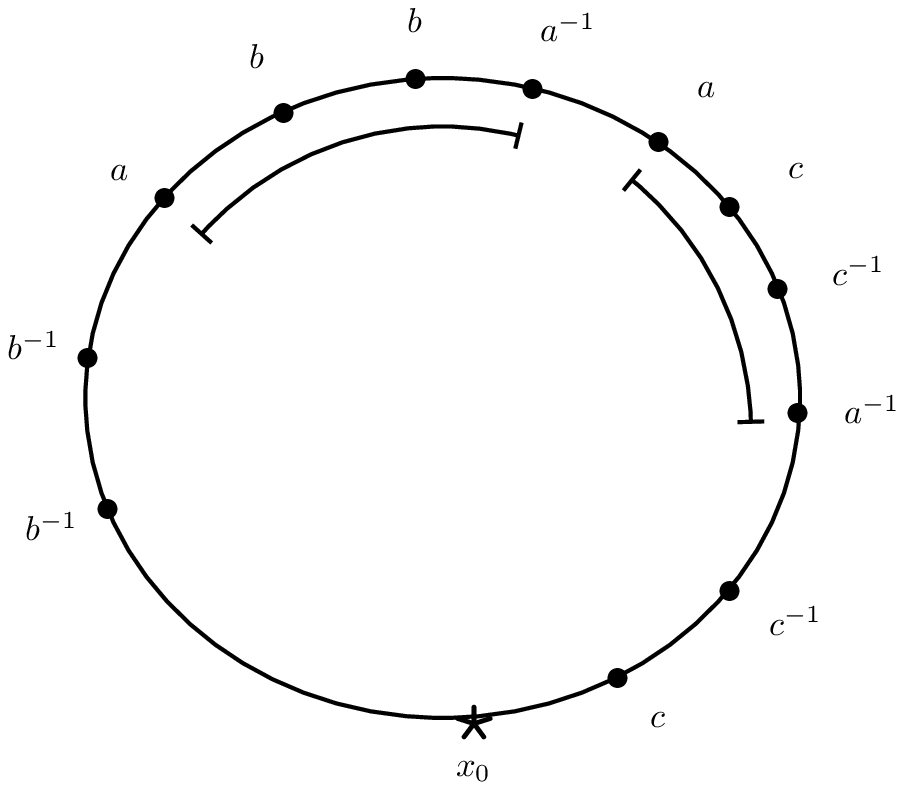}
\hspace{2.3 cm}
\includegraphics[width=5cm]{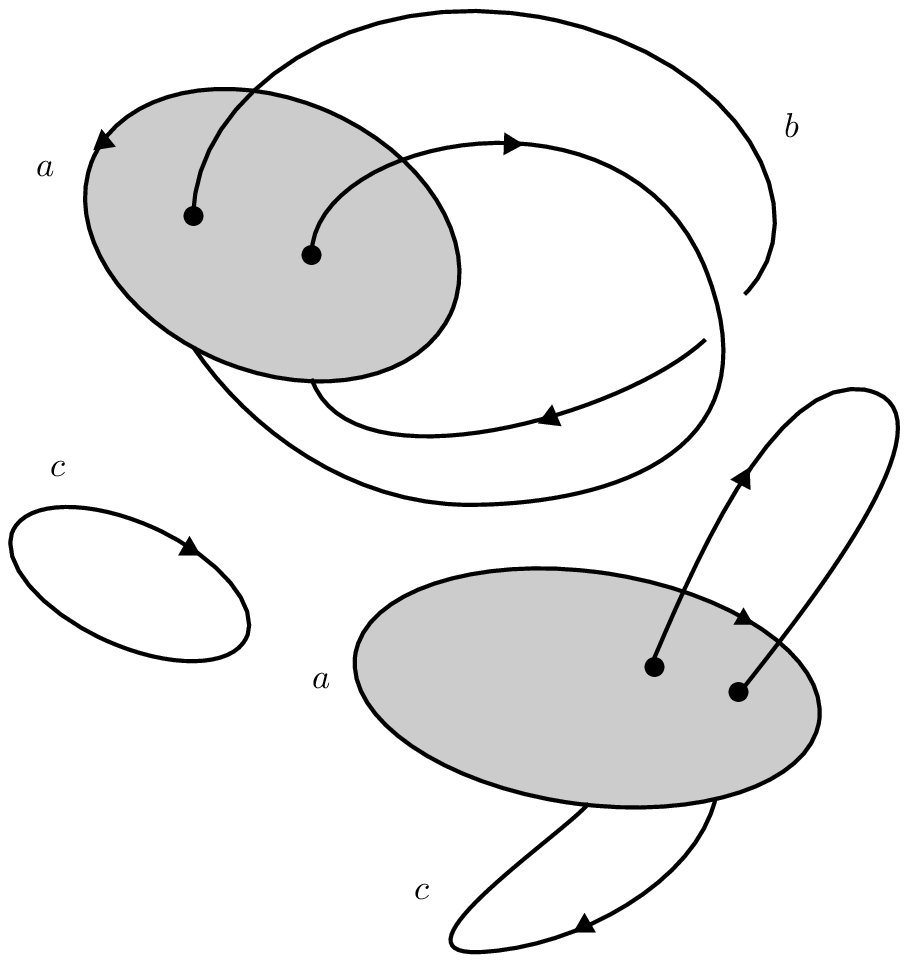} 

\caption{Classical linking is on the right ; linking of letters on the  left.}
\label{introfigure}
\end{figure}

 In both the left and right depicted cases in Figure~\ref{introfigure}, the distinct submanifolds are labeled and oriented, 
and linking number is calculated by cobounding the $a$ manifold and finding  two intersections with the $b$ manifold and two intersections of the $c$ manifold.  In both cases,  the $b$ intersections are both
positive, while the $c$ intersections have opposite signs.
 
 As also illustrated on the left,
a labeled, oriented zero-dimensional submanifold of a connected, oriented one-manifold with basepoint
corresponds to a word in the labels.   Under this correspondence,
 choosing a cobounding disk is equivalent to choosing 
pairs of occurrences  of a letter and its inverse and considering the letters between them.  The linking number counts instances of other letters in between.  We call such counts linking or 
interleaving numbers of letters, the subject we initiate in this paper.


Linking, and its constituent processes of cobounding  and intersecting, are staples in topology.
The process of cobounding and intersecting can be expanded and iterated to obtain higher linking numbers. 
Sinha and Walter show in \cite{SiWa13} that higher linking numbers ``with correction terms'' can be used to distinguish rational homotopy groups of a 
simply connected space.
The analogous algebraic generalization, obtained by considering the fundamental group of a space,
is  straightforward to conceptualize, at least in the free group setting.  
For example, the word $a b a^{-1} b^{-1}$  is zero in
the abelianization, and visibly is in the first commutator subgroup.  It has linking number one, as there is a single $b$ ``caught between'' an $a$-$a^{-1}$ pair.
We will see that this linking number obstructs its being in the second commutator subgroup.  In a further example, all of the two-letter
linking invariants of $[[a,b],c] = a b a^{-1} b^{-1} c b a b^{-1} a^{-1} c^{-1}$ vanish.  But if we then consider occurrences of $c$ between both an 
$a$-$a^{-1}$ pair and a $b$-$b^{-1}$ pair, that count is non-zero. We will see that this count,  which is  
modeled on a way to distinguish maps from a four-sphere
to a wedge sum of three two-spheres, obstructs this word from being a three-fold commutator.

Our main results are to define purely algebraic linking and higher linking numbers between letters of words, 
and show that they perfectly reflect the lower central series filtration of free groups, spanning the rational linear duals of their subquotients.   While
we could argue geometrically, using zero- and one-dimensional manifolds as in the discussion above, we take a purely 
algebraic and combinatorial approach.  
The subquotients of the lower central series for free groups constitute free Lie algebras, whose bases and linear duals have been
studied  extensively \cite{Reut93, SiWa11, MeRe96, BoCh06, Chib06, Walt20}.  
We lift the definition of functionals from the free Lie algebra subquotients to the free groups themselves, as understanding the
equivalence class in the subquotient is the goal rather than the starting point.  
Such functionals at the group level were given first implicitly by Magnus \cite{Magn37} and then explicitly through the
free differential calculus by Chen, Fox and Lyndon \cite{CFL58}.  We compare our approach to those, showing that it is  more efficient 
in examples,  is more closely related to the Quillen models for rational homotopy theory, and that it presents a
a new basis for the cofree Lie coalgebra.

We conjecture similar results for the linear dual to the rational lower central series Lie algebra of any 
finitely presented group.  We expect such results to be useful in a wide range of settings, including in topology in the presence
of the fundamental group.  The first author is pursuing such an application to the Johnson filtration of mapping class groups.
Together we are also pursuing an ``Ouroboros-like'' application to knot and link invariants defined 
through the fundamental group, in particular Milnor invariants
\cite{Mil54}.


\tableofcontents

\section{First development}

 \subsection{Basic definitions}

Our definitions are modeled on linking numbers of $S^0$ in $S^1$, with the $S^0$'s corresponding
to pairs of letters in a word and the cobounding intervals corresponding to sets of consecutive letters.  The objects 
being linked and the cobounding objects can be defined either through subsets of a word or through functions on the letters
of a word.  Both descriptions are useful,  more so together.  

As we introduce a fair number of basic definitions, some readers may want to refer to examples in Section~\ref{examples} as they 
go along, or even try to understand that approach at a conceptual level before reading the formal treatment here. 

 \begin{definition}
 Let $F_n$ denote the free group on $n$ generators.
 
 A {\bf word} of {\bf length} $k$ is an element of the $k$th Cartesian power of the set of generators and their inverses, for some $k \geq 0$.  Each entry of
 this Cartesian product is called a {\bf letter} of the word.
 \end{definition}
 
We generally let $w=x_1\cdots x_k$ denote a word of length $k$, which represents an element of $F_n$ in the standard way.

\begin{definition}
 A signed {\bf Linking Invariant Set for Tallying} - in short, a  ``{\bf list}'' - 
taken or drawn from a word $w$ is a set, possibly empty and possibly with repetition, of pairs
\[
L = \{ (\ell_1,\epsilon_1),\ldots, (\ell_m,\epsilon_m) \}
\]
where each $\ell_i$ is some $x_j$ and each extrinsic sign $\epsilon_i$ is $\pm1$.  We call such pairs {\bf signed letters}.
\end{definition}

See Example~\ref{ex1} below.  
For any word, there is a tautological list taken from it, namely
the one in which each letter occurs once with an extrinsic sign of $+1$.

For many purposes, it is helpful to understand such lists through  functions on the word.

\begin{definition} 
Given a list $L=\{(\ell_1, \epsilon_1), \ldots, (\ell_m, \epsilon_m)\}$, 
the {\bf associated function $f_L$} from the ordered set $w = x_1, \ldots, x_k$ to the integers sends each $x_j  $ to 
the sum of the $\epsilon_i$ associated to its occurrences  as an $\ell_i$ in  $L$.  
Two lists are {\bf simply equivalent} if their associated functions are equal.
\end{definition}

Lists up to simple equivalence thus form a group, where addition is union and inverse is reversing all extrinsic signs.
Two lists are simply equivalent if and only they are related by a sequence of additions or removals of  pairs $\{(x_i, 1), (x_i, -1)\}$.

In our applications, lists  will almost always be homogeneous, consisting of occurrences of a single generator and its inverse.  In such cases
we sometimes  incorporate the generator which occurs in the name of the list, for example letting $L_a$ and $L'_a$ be lists
of the letter $a$.

\begin{definition}
A list is {\bf homogeneous} if every letter is an occurrence of the same generator or its inverse.

The {\bf standard list} $\Lambda_a$ of a generator $a$ in $w$  is formed 
by having each $x_i=a^{\pm 1}$ in $w$ appear in the list one and only one time, with extrinsic sign $+1$.
\end{definition} 

 The associated function 
for this standard list is the indicator function for the subset of occurrences of $a$ and $a^{-1}$.


As the acronym implies, we can tally or count a list.

\begin{definition}
Let $a$ be a generator, and $x$ either a generator or the inverse of a generator.  The {\bf intrinsic sign} with respect to $a$ is defined by

\begin{equation*}
 \sign^a(x)=
\begin{cases}
1 & \text{ if } x=a \\
-1 & \text{ if }x=a^{-1}\\
0 & \text{ else }. \\
\end{cases} 
\end{equation*}

The {\bf total sign} of  a signed letter with respect to a generator $a$, which by  we call $s_a(\ell_i)$ (or by abuse $s(\ell_i)$ when $a$ is understood), 
is $\epsilon_i \cdot  \sign^a(\ell_i)$.

The {\bf $a$-count} $\phi_a$ of a list $L =\{(\ell_1,\epsilon_1),\ldots, (\ell_m, \epsilon_m)\}$ is given by $$\phi_a(L)= \sum_{i=1}^m s(\ell_i).$$  
\end{definition}

\begin{example}\label{ex1}
 For $w=aaba^{-1}b^{-1}$, two  lists we can form are  $L_a=\{(a_1,1),(a_1,-1), (a^{-1},-1)\}$ and 
$L_a'=\{(a_1, 1),(a_2,1),(a^{-1},1)\}$.  Here since there are multiple occurrences of the letter $a$, we distinguish them by subscripts
which reflect their order of occurrence.
The first list is simply equivalent to $\{ (a^{-1}, -1) \}$, while the second list is $\Lambda_a$. 
We have that $\phi_a(L_a)=1$ and $\phi_a(L_a')=1$. 
\end{example}

If the letter which we are counting is understood in context -- in particular, if a list is homogeneous -- 
then we omit the letter and simply use $\phi$ for the appropriate counting function.

 The counts  for the tautological list associated to a word  with respect to all generators determine its image in the abelianization
 of the free group.  
 When these vanish, a word
represents an element of the commutator subgroup.  We now define derived counts in the commutator subgroup.  To do so, 
we first define cobounding.

\begin{definition}
An {\bf interval} $I$  in a word $w$ is a nonempty set of consecutive letters.  Such is determined by its {\bf first and last letters}, 
denoted $\partial_0 I$ and $\partial_1 I$.

An {\bf oriented interval} is an interval whose endpoints are signed letters with opposite total signs.
We let  the {\bf boundary} 
$\partial I = \{ (\partial_0 I, \sigma_0), 
(\partial_1 I, \sigma_1) \}$ 
where $\sigma_0$ and $\sigma_1$ are the extrinsic signs of $\partial_0 I$ and $\partial_1 I$ respectively.
We let $\epsilon^I_{0}$ be the total sign of  $\partial_0 I$, and similarly for $\epsilon^I_1$.
The {\bf orientation} of $I$, denoted ${\rm or} \; I$,  is defined to be $\epsilon^I_0$. 

An oriented interval determines an {\bf associated function} $f_I$ from the   set $x_1, \ldots, x_k$ to the integers whose value is $0$
except on the consecutive set of letters defining $I$ , and whose value on each letter in this set of consecutive 
letters  is the total sign of the initial letter.
\end{definition}


In Section~\ref{examples}, we will give a diagrammatic approach to doing these counts by hand.  
In that framework, the orientation of an interval 
proceeds from the positively signed endpoint to the negatively signed endpoint.

\begin{definition}
Let $L=\{(\ell_1,\epsilon_1),\ldots, (\ell_p,\epsilon_p)\}$ be a nonempty homogeneous list with 
$\phi(L)=0$.  

Define a {\bf cobounding} $d^{-1}L$ to be a  set of oriented intervals $\{ I_k \}$ such that each 
$(\ell_i, \epsilon_i)$ occurs exactly
once as either $\partial_0$ or $\partial_1$ of some $I_k$. 
\end{definition}

One can use the  ordering on the letters of a word to define a canonical cobounding, but we will make use of the flexibility in such choices.
Cobounding seems awkward to define through functions.

We now define linking of letters.  Since linking is intersection with a choice of cobounding, the path forward is clear.

\begin{definition}\label{intersection}
Let $w$ be a word, $L_a$ a list with $\phi (L_a) = 0$,  and  let  
$d^{-1}L_a =  \{ I_k \}$  be a choice of a cobounding.  Let $L_b  = \{ (y_1, \epsilon_1),\ldots, (y_p,\epsilon_p) \}$ be a list  with $b\neq a$.  

Let $(x_i, \epsilon)$ be a signed generator and
first define the {\bf signed intersection} $(x_i, \epsilon) \cap I_k$ to be either $(x_i,   \epsilon_0^I \cdot \epsilon)$ if $x_i \in I_k$  or empty if $x_i \notin I_k$.  
Define the {\bf linking list}  $d^{-1}L_a \wedge L_b$, which is also by definition $L_b \wedge d^{-1} L_a$, as the union of all 
$(y_i,\epsilon) \cap I_k $, as $(y_i,\epsilon)$ varies over the elements of $L_b$ and $I_k$ varies over the intervals in $d^{-1}L_a$.

\end{definition}

We will shortly provide an illustration, in Example~\ref{graphical}.

The simple equivalence class of the list $d^{-1}L_a \wedge L_b$, is that whose associated function is
 $\sum_{I \in d^{-1} L_a} f_{L_b} \cdot f_I$.  
The list $d^{-1}L_a \wedge L_b$ witnesses the linking of the lists $L_a$ and $L_b$.
As $d^{-1}L_a \wedge L_b$ is itself a list, of generator $b$, we can count it.  We can then 
repeat the process whenever the count of a list which 
is produced vanishes.  Recall the standard list $\Lambda_a(w)$.  Then for example if $d^{-1} \Lambda_a \wedge \Lambda_b$ is defined and
 has vanishing count,  we can then define $d^{-1} ( d^{-1} \Lambda_a \wedge \Lambda_b ) \wedge \Lambda_\ell$, where $\ell$
could be any generator other than $b$.

\begin{definition}\label{levellist}
An {\bf iterated linking list} is one obtained from the lists $\Lambda_\ell$ through the operations of (choice of) cobounding and intersection, when defined.

Define the {\bf depth} of $\Lambda_a$ to be zero.
Inductively, if $L_a$ is an iterated linking list of depth $i$ of $w$ with $\phi (L_a) = 0$ and $L_b$ is of depth $j$ with $b \neq a$ then we
define the depth of a list $d^{-1}L_a \wedge L_b$ to be $i + j + 1$.
\end{definition}

We will count these iterated linking lists, informally calling the results the  ``letter-linking'' or ``interleaving'' 
numbers for the word $w$.  While cobounding involves a choice,
one of our first results is that such choices, at any stage, do not change resulting counts.  But before  establishing such needed results,
we share an example.

\subsection{An example, through a visual algorithm}\label{examples}

The geometric inspiration for our letter-linking numbers gives rise to a visual algorithm for computing them  by hand.  
First we  describe  geometric representations for the definitions above, and then we will give an illustrative example.

 Let $w$ be a word  with lists $L_a$ and  $L_b$ taken from it, with  $\Phi_a(L_a)=0$.  
 We construct a choice of $d^{-1}L_a \wedge L_b$ through a diagram, starting with $w$ and decorating it in the following  steps:

\begin{enumerate}
\item For each $(a, \epsilon)$ in $L_a$,  place a $+$ sign above the corresponding $a$ in $w$ if $\epsilon=1$ and a minus sign if $\epsilon=-1$.
Alternatively, just list total multiplicities over the letters.

\item Each interval chosen for $d^{-1}L_a$  corresponds to a choice of elements with opposite total sign.  For each, 
draw an arrow originating under the letter with positive total sign and 
ending under the element with negative total sign.

\item As in the first step, decorate each $b$ in $w$ with  its multiplicity in $L_b$, say $m$ which is the positive occurrences minus the negative occurrences of this $b$ in the list.

\item  For each $b$ consider the arrows which ``pass through'' it.  
If there are $p$ arrows heading left to right and $q$ arrows heading from right to left which cross that occurrence of $b$, 
then replace its multiplicity by $m(p-q)$.  In particular, if there are no arrows  underneath the letter replace its multiplicity by $0$.

\end{enumerate}

These multiplicities over all $b \in L_b$ are the associated function for $d^{-1}L_a \wedge L_b$.

\bigskip

\begin{example}\label{graphical}
 Consider   ${d^{-1}(d^{-1}\Lambda_a\wedge\Lambda_b)\wedge \Lambda_a}(w)$ for $w=[aa,[b,ac]]$.  
 We first verify by inspection that all letter-linking numbers of depth one vanish.  
Then we diagram $d^{-1}\Lambda_a \wedge \Lambda_b$ as
\begin{center}
\includegraphics{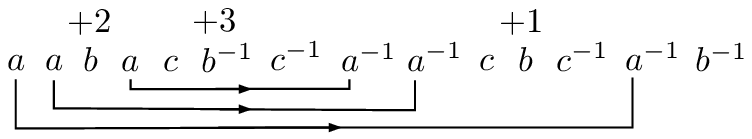}.
\end{center}
We find that $L_b=d^{-1}\Lambda_a \wedge \Lambda_b$  has two occurrences of the first $b$, three of the first $b^{-1}$ and 
one occurrence of the second $b$.

We can then diagram $d^{-1}L_b \wedge \Lambda_a$ as follows,

\begin{center}
\includegraphics{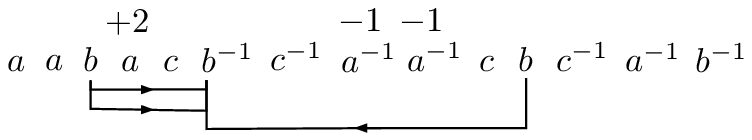}.
\end{center}
We deduce $\phi_a ({d^{-1}(d^{-1}\Lambda_a\wedge\Lambda_b)\wedge \Lambda_a}(w))  =2\cdot \sign(a)-1\cdot\sign(a^{-1})-1\cdot \sign(a^{-1})=4$.   
While $w$ is visibly a two-fold commutator, we will see that the count just made obstructs $w$ being a three-fold commutator.

\end{example}



\subsection{Independence from choices}\label{independence}

\begin{definition}
The {\bf provisional symbol} $\sigma$ of a depth-$i$ list is its expression as an iterated application of $d^{-1}$ and $\wedge$ to lists $\Lambda_\ell$.  

 Define the letter-linking function  $\Phi_\sigma(w)$ to be  $\phi ( L)$, 
where $L$ is a choice of list with symbol $\sigma$, when such a list exists.
\end{definition}


The main result of this paper is  that the letter-linking functions $\Phi_\sigma$ 
determine the representative of a word in the lower central series
Lie algebra of a free group.
In the next sections we will develop relations between these functions,  connect them to this lower central 
series filtration, and prove this main result.  
But  we must first prove that these are in fact well-defined functions, which we do now in steps.
We will show that the choices in cobounding and in representative of an element of the free group
result in simply equivalent lists, which agree not only in their counts but in all of their derived counts.
This proof, and others below, rely on the geometry of intervals.

\begin{definition}
We say two intervals in a word are 
 {\bf disjoint} if they have no letters in common. We say they are  
 {\bf contained}  if one is contained in the other.
 Otherwise,
we say they are {\bf interleaved}.

Define an {\bf exchange} of intervals to be replacing two intervals in a cobounding with two different intervals with the same four boundary points.  
\end{definition}

\begin{proposition}\label{indepdinv}

 Let $w$ be a  word, and let $L_a$ and $L_b$ be lists taken from 
 $w$ with $\phi(L_a) = 0$.
 Then  all choices of $d^{-1}L_a \wedge L_b$  are simply equivalent.  
 
 \end{proposition}

\begin{proof} 

Any two 
coboundings differ by a sequence of exchanges, so we analyze a single exchange.
We claim that after performing an exchange in $d^{-1} L_a$, intersecting with $L_b$ yields 
a list which can only differ by  including or omitting  pairs of the form $(\ell,-1)$ and $(\ell, 1)$.

An exchange can occur between any two types of intervals.  We focus on the case of exchanging between disjoint and interleaved.  Similar
arguments establish the other cases.

Consulting   Figure~\ref{Figure 2}, let $v,x,y,z$ be the letters in the word $w$, in the order in which they occur, which are the endpoints of the exchanged intervals.
 If there is an exchange between disjoint and interleaved coboundings of $v, x, y$ and $z$ then $x$ and $y$ must have the same
total sign.  Thus the two intervals in each matching have opposite orientations, and the leftmost intervals in the disjoint and interleaved coboundings
have the same orientation, as do the
rightmost.

\begin{figure}[h]
\includegraphics[width=6cm]{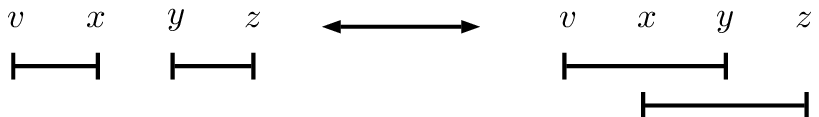}
\caption{An  exchange between disjoint and interleaved intervals.}
\label{Figure 2}
\end{figure}

In this case, occurrences of $b$ between $v$ and $x$ 
and those between $y$ and $z$ are added once to the list $d^{-1}L_a \wedge L_b$ for both choices of cobounding, with the same signs.  
The occurrences of $b$ between $x$ and
$y$ do not get added at all for the disjoint cobounding, and in pairs with opposite signs for the interleaved cobounding.  Thus the lists differ by  pairs 
$(b, 1)$, $(b,-1)$, and so are simply equivalent.
\end{proof}

\begin{corollary}\label{simplyequiv}
 If  $L_a$ and $L_a'$ are simply equivalent, as are $L_b$ and $L_b'$, then so are $d^{-1} L_a \wedge L_b$ and $d^{-1} L_a' \wedge L_b'$,
for any choice of coboundings.  
\end{corollary}

\begin{proof} It suffices to consider $L_a$ and $L_a'$ which differ by one cancelling pair, say with $L_a'$ having the additional pair.
A  choice  $d^{-1}L_a$ of cobounding for $L_a$ can be extended  
to one for $d^{-1}L_a'$,  by adding the cancelling pair as an interval in the set.  But no other letter including $b$ can intersect this interval, 
so with these choices the resulting lists are the same.  That $d^{-1} L_a \wedge L_b$ and $d^{-1} L_a \wedge L_b'$ differ
by cancelling pairs when $L_b$ and $L_b'$ do is immediate.  
Applying Proposition~\ref{indepdinv}, the resulting lists will be simply equivalent for any choices
of coboundings.
\end{proof}

We can now prove well-definedness of these letter interleaving numbers.

\begin{theorem}\label{well-defined}
The function $\Phi_\sigma$ is independent of choice of list with symbol $\sigma$ and is independent of word representative of group element.
\end{theorem}

\begin{proof}
First fix a word representative $w$.  That  $\Phi_\sigma(w)$ is independent of choice of list with symbol $\sigma$  
is immediate through inductive application of Corollary~\ref{simplyequiv}, 
which implies that all lists with the  symbol $\sigma$ will be simply equivalent
and thus have the same count $\Phi_\sigma$.

To show that the functions are well-defined on the free group, consider $w=w_1w_2$ and ${w}' =w_1 aa^{-1} w_2$.
We  identify lists in $w$ with lists in ${w}'$ and inductively show that there are choices of 
depth-$i$ lists on ${w}'$ which differ by consecutive pairs -- that is unions of the set $\{ (a, \varepsilon), (a^{-1}, \varepsilon) \}$ where $a$ and $a^{-1}$
are the added pair in ${w}'$ --  from the  lists 
with the same symbol on $w$.

The base case of lists $\Lambda_\ell$ is immediate.
Consider  some $d^{-1} L_a \wedge L_b$.  By inductive hypothesis, $L_a$ on ${w}'$  differs from the $L_a$
with the same symbol on $w$ by consecutive pairs.  Choose a cobounding which starts by taking intervals whose endpoints are 
consecutive pairs before cobounding the rest of the list.  Since no $b$'s can be in
the consecutive pair intervals, the lists $d^{-1} L_a \wedge L_b$  will be the same.
Next for $d^{-1} L_b \wedge L_a$ the lists will differ by consecutive pairs, as intervals cannot have their endpoints between 
$a$ and $a^{-1}$.  Finally, any $d^{-1} L_b \wedge L_c$ for $b,c$ distinct from $a$ will not differ between $w$ and ${w}'$. 
\end{proof}



\section{Symbol notation and basic identities}

\subsection{Symbols}

We first  develop more compact notation for symbols, replacing  $d^{-1}$ with parentheses, $\wedge$ with juxtoposition,
and $\Lambda_\ell$ with $\ell$.

\begin{definition}
A {\bf pre-symbol} is a parenthesized word in a generating set (no inverses) such that
\begin{itemize}
\item There is exactly one fewer pair of parentheses than letters.
\item  Every pair of parentheses contains exactly one letter which is not further parenthesized, which we call its free letter.
\item Every pair of parentheses is either nested or disjoint.

\end{itemize}
\end{definition}

The third condition disambiguates repeated parentheses in the standard way.
The first two conditions imply that at least one
single letter is parenthesized by itself and one letter is unparenthesized, within any pair of parentheses with at least two letters as well as for the entire word.
Examples include   $a(b(c))$, $(a) b (c)$,  $(a(e)) (a) (c) b$, and $((a)(a)b)c$.   

\begin{definition}

The {\bf depth} of a pre-symbol is the number of pairs of parentheses (one less than the length of the word).

A {\bf sub-(pre-)-symbol} of a pre-symbol  consists of either a  letter or a pair of parentheses and its contents.  The sub-(pre-)symbols
of a pre-symbol form a poset under {\bf containment}.  
Two symbols are  {\bf equivalent} if  their containment posets are isomorphic, through an isomorphism which preserves labels.
\end{definition}


Thus for example $((a)(a)b)c$ is equivalent to $c((a)b(a))$.  
Alternatively, equivalent  symbols can be obtained from one another by a series of permutations of the immediate contents
of any pair of parentheses, as well as permutation of outermost sub-symbols.  (One could have an  outmost set of parentheses 
to make this and other aspects of symbols more uniform, but
we have chosen not to as the current definition seamlessly fits with our letter-linking definitions.)

Symbols are 
reminiscent of parenthesizations in non-associative algebras such as Lie algebras, and the containment poset of a symbol is a rooted tree.  
We will develop  a duality with Lie brackets, but this duality is not
based on such superficial similarity.

\begin{definition}
The shortened (pre-)symbol of a depth-$j$ list is obtained inductively as follows.
\begin{itemize}
\item The shortened symbol of $\Lambda_\ell$ is $\ell$.
\item If the shortened symbol of $L_a$ is $\sigma$ and that of $L_b$ is $\tau$ then the shortened symbol of $d^{-1} L_a \wedge L_b$ is $(\sigma) \tau$ 
-- that is, the shortened symbol of $L_a$ parenthesized and followed by that  of $L_b$.  
\end{itemize}
\end{definition}

Equivalent symbols will give rise to isomorphic lists.
From now on, we use shortened (pre-)symbols to describe letter-linking invariants.  

Recall that in our definition of linking of lists, 
the lists in question must 
be comprised of different letters.  We capture this condition as follows.

\begin{definition}
Consider a pair of parentheses in a pre-symbol, whose immediate contents are of the form $(\sigma_1) \ldots (\sigma_k) \ell$, 
where each $\sigma_i$ is a symbol
and $\ell$ is the free
letter.  An almost-symbol is {\bf valid}  for this pair of parentheses if the free letters of $\sigma_i$ are all different from $\ell$.  
Define a {\bf symbol} to be pre-symbol which is valid for all its pairs of parentheses.
\end{definition}

In other words, free letters may be repeated, but not at neighboring levels. 
 
 \subsection{Homomorphism identities} 
 
In Section~\ref{independence} we showed that the letter-linking functions $\Phi_\sigma$ are well
defined, but we should recall that they are only defined on subsets of the free group, 
since the definition of $d^{-1} \Lambda_\mu \wedge \Lambda_\tau$
requires the vanishing of  $\Phi_\mu$.  
 Always implicit in our statements of identities in this paper, in particular those in this section, is that equalities hold
 only where all quantities involved are defined.  
 
 Recall that multiplication in free groups corresponds to concatenation of words, which we  denote by $w_1 \cdot w_2$.

\begin{proposition}
$\Phi_\sigma(w_1 \cdot w_2) = \Phi_\sigma(w_1) + \Phi_{\sigma}(w_2)$.
\end{proposition}

\begin{proof}
Inductively apply two facts.  First, when all defined, the coboundings on $w_1 \cdot w_2$ can be chosen to be the union of those on $w_1$ and $w_2$.
Secondly,  $\wedge$ distributes over union of coboundings.
\end{proof}

Recall that taking inverses in free groups corresponds to reversing the letters in a word and changing all intrinsic signs to their opposite, 
which we denote by $w^{-1}$.

\begin{proposition} 
$\Phi_\sigma(w^{-1}) = -\Phi_\sigma(w)$.
\end{proposition}

\begin{proof}
Define compatible involutions on lists and their coboundings by reversing letters and taking inverses, but leaving extrinsic signs unchanged.
Inductively we can choose $\Lambda_\sigma(w^{-1})$ as the image
of $\Lambda_\sigma(w)$ under this involution.  Under this involution,  counts of lists are multiplied by $-1$, as the orientations of
corresponding intervals
will not change (as they will change twice) while the signs of the corresponding letters will change.
\end{proof}


The homomorphisms $\Phi_\ell$, for a generator $\ell$, are  the composite of the map from the free group to its abelianization followed by
projection onto the $\ell$-summand of the abelianization.  We view the other homomorphisms
$\Phi_\sigma$ as derived from this abelianization.

\subsection{Leibniz identities}

 The geometry of intervals gives rise to key relations.
 
\begin{definition}
The {\bf intersection} $I \cap J$ of two oriented intervals is their intersection, with orientation given by the product of orientations.

If $S_i$, for $i = 1, \ldots, n$ are sets of intervals, define $\bigcap S_i$ to be $\bigcup_{I_1 \in S_1, \cdots, I_n \in S_n} I_1 \cap \cdots \cap I_n$.

\end{definition}

The following two facts are immediate from the definitions.

\begin{lemma}[Associativity]
$(I \cap J) \cap K = I \cap (J \cap K)$ and 
$((x_i, \epsilon) \cap I) \cap J = (x_i, \epsilon) \cap ( I \cap J ).$
\end{lemma}

\begin{proposition}[Leibniz Rule]
$\partial (I \cap J) = (\partial I \cap J) \cup (I \cap \partial J$).  More generally 
$$\partial \bigcap_{i = 1 \cdots n}  S_i = \bigcup_i S_1 \cap \cdots \cap S_{i-1} \cap \partial S_i \cap S_{i+1} \cap \cdots \cap S_n.$$
\end{proposition}

Here we are using  $\cap$ from Definition~\ref{intersection} for intersecting sets of signed letters with lists.

\begin{proposition}\label{arnoldgen}(Leibniz Relation)
$$\Phi_{\sigma_1 (\sigma_2) \cdots (\sigma_{k-1}) (\sigma_{k})} 
+  \Phi_{(\sigma_1) \sigma_2  \cdots (\sigma_{k-1}) (\sigma_{k})} + \cdots 
+  \Phi_{(\sigma_1) (\sigma_2)  \cdots \sigma_{k-1} (\sigma_{k})} 
+  \Phi_{(\sigma_1) (\sigma_2)  \cdots (\sigma_{k-1}) \sigma_{k}} = 0.$$ 
\end{proposition}

The first two cases of this identity have distinct names.  The $k = 2$ case, which can be rewritten as 
$\Phi_{(\sigma) \tau} = -\Phi_{\sigma (\tau)}$, is known as an antisymmetry relation.  
It should be considered in contrast with the fact that by definition, or essentially by commutativity of intersection,
 $\Phi_{(\sigma) \tau} = \Phi_{\tau (\sigma)}$.  We call the $k=3$ case the Arnold identity, with the connection to the 
 identity with the same name in topology, which figures prominently in our work in Section~\ref{mainsection}. 

\begin{proof}[Proof of Proposition~\ref{arnoldgen}]
The relation follows from a slightly more general fact.  Let $a_1, \cdots, a_k$ be distinct letters, and $d^{-1} L_{a_i}$ coboundings of lists in those letters.
Because $\partial \bigcap_i d^{-1} L_{a_i}$ is the boundary of a collection of intervals, its count is zero.  Thus by the Leibniz rule
$$\phi(\bigcup_i (d^{-1}L_{a_1} \cap d^{-1}L_{a_2}  \cap \cdots \cap {\widehat{d^{-1}L_{a_i}}} \cap \cdots \cap d^{-1}L_{a_k}) \wedge L_{a_i}) = 0.$$
Setting $L_{a_i}$ to be $\Lambda_{\sigma_i}$ in this equality establishes the Leibniz Relation.
\end{proof}

\subsection{Commutator identities}

We start with a result which could independently be deduced from stronger results in the next section.  The current treatment motivates those
results and illustrates  the proof technique for the stronger Theorem~\ref{definedonlcs} below.
 
 
 \begin{proposition}\label{cobracket1}
 $\Phi_{(\sigma) \tau}  [v, w] = \Phi_{\sigma}(v) \Phi_\tau(w) -  \Phi_\tau(v) \Phi_\sigma(w).$
 \end{proposition}

\begin{proof}
By convention $\Phi_\sigma(v)$ and $\Phi_{\sigma}(w)$ are defined, and thus so are $\Phi_\sigma(v^{-1})$ and $\Phi_{\sigma}(w^{-1})$.
We use the  lists and cobounding intervals which define them to produce the list $\Lambda_\sigma(vwv^{-1}w^{-1})$ as the union of 
lists identified with $\Lambda_\sigma(v)$, $\Lambda_\sigma(w)$, $\Lambda_\sigma(v^{-1})$, and $\Lambda_\sigma(w^{-1})$, which
we also assume to be chosen to respect the inverse involution.
Call the generator in these lists $a$.  The inverse involution matches occurrences of 
$a^{\pm 1}$ in $\Lambda_\sigma(v)$ with those
of $a^{\mp 1}$ in $\Lambda_\sigma(v^{-1})$ and similarly for $w$, $w^{-1}$, through which we choose our cobounding $d^{-1} \Lambda_\sigma([v,w])$.
  Similarly choose $\Lambda_\tau([v,w])$ as the union of lists which can be identified with 
  $\Lambda_\tau(v)$, $\Lambda_\tau(w)$, $\Lambda_\tau(v^{-1})$, and $ \Lambda_\tau(w^{-1})$.  See Figure~\ref{secondfigure} for a schematic.
  
\begin{figure}[h]
\includegraphics[width=10cm]{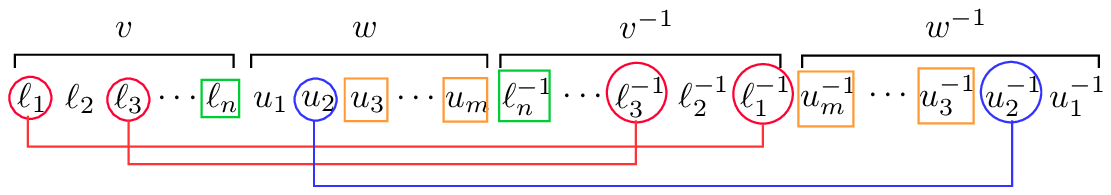} 
\caption{Schematic  for Proposition~\ref{cobracket1}.  \\
 $\Lambda_\sigma(v)$ \& $\Lambda_\sigma(v^{-1})$ are denoted by red circles;   
$\Lambda_\sigma(w)$ \& $\Lambda_\sigma(w^{-1})$ by blue circles; $\Lambda_\tau(v)$ \& $\Lambda_\tau(v^{-1})$ by  green squares;
 $\Lambda_\tau(w)$ \& $\Lambda_\tau(w^{-1})$ by orange squares. }\label{secondfigure}
\end{figure}
   
With these choices  consider  $d^{-1} \Lambda_\sigma([v,w]) \wedge  \Lambda_\tau([v,w])$, starting with the intervals which cobound
across $v$ and $v^{-1}$.  Such intervals do not intersect $\Lambda_\tau(w^{-1})$.  Their intersections with $\Lambda_\tau(v)$ and 
$\Lambda_\tau(v^{-1})$ are matched under
the inverse involution we have used to define our lists.  And these intervals have full intersections with $\Lambda_\tau(w)$, yielding a total contribution
of $\Phi_\sigma(v) \cdot \Phi_\tau(w)$ to $\Phi_{(\sigma) \tau}  [v, w]$.
  
Analysis of the coboundings across $w$ and $w^{-1}$ are similar, with only intersections 
with $\Lambda_{\tau}(v^{-1})$ contributing, yielding $ -  \Phi_\tau(v) \Phi_\sigma(w).$
\end{proof}

\begin{definition}
Let $G$ be a group.  Inductively define the lower central series of groups by $\gamma_i G = [\gamma_{i-1} G, G]$, with $\gamma_0 G = G$.
Inductively define the derived groups by $G^{(i)} = [G^{(i-1)}, G^{(i-1)}]$,  with $G^{(0)} = G$.
\end{definition}

We may inductively apply Proposition~\ref{cobracket1}, starting with the immediate fact that $\Phi_\ell$  vanishes on commutators, to see that 
$\Phi_\sigma$ vanishes on the derived group $F_n^{(i)}$ for $i$  greater than the depth of $\sigma$.
But we  now show that the letter-linking functions $\Phi_\sigma$  in fact vanish on the  lower central series.  


\begin{theorem}\label{definedonlcs}
If the depth of a symbol is less than $i$, then the corresponding letter-linking invariant vanishes on $\gamma_i F_n$.  Thus if 
the depth of $\sigma$ is equal to $i$, $\Phi_\sigma$ is defined on $\gamma_i F_n$.
\end{theorem}

\begin{proof}
We argue inductively, starting with the immediate fact that $\Phi_\ell$ vanishes on commutators.

Let $\sigma$ be a symbol of depth $i-1$.  Let $w \in \gamma_{i-1} F_n$ and $v \in F_n$, so $[w^{-1}, v]$ represents
an element of $ \gamma_i F_n$, 
and consider $\Phi_\sigma ([w^{-1} ,v]) = \Phi_\sigma ( w^{-1} v w v^{-1} )$.  As $\Phi_\sigma$ is defined on $w$ and $w^{-1}$ we choose
to construct $\Lambda_\sigma [w^{-1},v]$ starting with $\Lambda_\sigma w$ and $\Lambda_\sigma w^{-1}$, whose counts  will cancel.

We complete $\Lambda_\sigma [w,v]$ through a construction of  $\Lambda_\sigma v w  v^{-1} \backslash \Lambda_\sigma w$,
which allows us to show the  count vanishes.  
Such a construction is the ultimate  case of a second inductive claim that for any symbol $\tau$ of 
depth less than or equal to $i-1$, the list 
$\Lambda_\tau v w v^{-1} \backslash \Lambda_\tau w$ is defined and is comprised of letters with particular forms in
$v \cup v^{-1}$ and $w$, namely:
\begin{itemize}
\item The letters in $v \cup v^{-1}$   are preserved under the $v$-$v^{-1}$ involution.
\item The letters in $w$  are a union of  lists $\Lambda_{\hat{\tau}_k} w$ for some collection of symbols 
$\{ \hat{\tau}_k \}$ each of depth less than that of $\tau$.  
\end{itemize}
For brevity we call lists of letters in $v \cup v^{-1}$ and in $w$ with these properties type-one and type-two, respectively. 
Type-one pairs immediately have zero count, and type-two letters have zero count as well since each 
$\Lambda_{\hat{\tau}_k} w$ has zero count by our 
primary inductive hypothesis.  

We prove this second claim through an induction on the depth of $\tau$.  
We may use the vanishing statement of our theorem  up to depth $i - 2$.
For depth zero, $\Lambda_\ell v w v^{-1} \backslash \Lambda_\ell w$ consists of occurrences of $\ell$ in $v v^{-1}$, which 
indeed occur in pairs preserved under involution, and thus is comprised entirely of type-one letters.  

Next assume $\tau = (\mu_1) \mu_2$ where the claim has been verified for the $\mu_i$.  We choose the cobounding of the type-one subset of 
$\Lambda_{\mu_1} v w v^{-1} \backslash \Lambda_{\mu_1}w$ by cobounding pairs which correspond with one another under involution.
We then cobound the type-two pairs, a cobounding which exists by inductive assumption because the depth of $\mu_1$ is less than $i-1$.
Consider the four cases for intersection arising in
 $(d^{-1} \Lambda_{\mu_1} v w v^{-1} \backslash d^{-1} \Lambda_{\mu_1} w) \wedge 
(\Lambda_{\mu_2} v w v^{-1} \backslash  \Lambda_{\mu_2} w)$:

\begin{itemize}
\item The intersection of type-one pairs in $\Lambda_{\mu_2} v w v^{-1} \backslash \Lambda_{\mu_2} w$ with  type-one
cobounding intervals in $d^{-1} \Lambda_{\mu_1} v w v^{-1} \backslash  d^{-1} \Lambda_{\mu_1} w^{-1}$ 
is a collection of type-one pairs.
\item The intersection of a type-one cobounding interval for $\mu_1$ with any list of letters in $w$, in particular any of
the $\Lambda_{\hat{\mu_2}_k w}$, is the list itself.  Thus the intersection of all such cobounding intervals  with a union of 
 $\Lambda_{\hat{\mu_2}_k w}$ is another such union.  Because the depth of $\hat{\mu_2}_k$ is less than that of $\mu_2$ it is less
 than that of $\tau$.
 \item Type-two cobounding intervals are contained in $w$, so their intersection with type-one pairs is empty.
 \item The intersection of type-two cobounding intervals from some $d^{-1} \Lambda_{\hat{\mu_1}_j} w$ with all the type-two pairs
 from some $\Lambda_{\hat{\mu_2}_k} w$ is by definition $\Lambda_{(\hat{\mu_1}_j) \hat{\mu_2}_k} w $.  Thus the union of all
 such intersections is the union of lists $\Lambda_{\hat{\tau}_\alpha} w$.
\end{itemize}

With this second induction step and thus the second induction claim established, we apply it for $\tau = \sigma$.  We deduce
that the count of $\Lambda_\sigma v w v^{-1} \backslash \Lambda_\sigma w$ is zero, completing our main induction.
\end{proof}

\begin{corollary}\label{lcsdefined}
The $\Phi_\sigma$ of depth $i$ are well defined on the lower central series subquotients $\gamma_i F_n / \gamma_{i+1} F_n$.  
\end{corollary}




\section{Lie coalgebraic graphs}\label{mainsection}

To evaluate letter interleaving invariants on the lower central series subquotients, 
 and in particular show they span the linear
dual,  it is necessary to 
bring in the combinatorial approach to Lie coalgebras of \cite{SiWa11, SiWa13}.
We give a logically self-contained treatment of the needed parts of this theory here.
We have a combinatorial rather than algebraic 
emphasis, but the motivation is still drawn from the topology encoded by 
Lie coalgebras and the geometry of Hopf invariants.

\subsection{Eil graphs and letter-linking}

 Recall that the free Lie algebra on a set $S$
 are well-known to be modeled by trees with leaves labeled by $S$. We found in \cite{SiWa11} that its linear dual, 
the cofree Lie coalgebra on $S$, has a natural model defined by acyclic graphs with vertices labeled by $S$, with a combinatorially rich 
(in particular, not Kronecker)
pairing between these models.
By Theorem~\ref{definedonlcs} and the fact that the lower-central series subquotients of free groups naturally form free Lie algebras,
our letter-linking homomorphisms are naturally identified with elements of  co-free Lie coalgebras.  But these homomorphisms are
indexed by symbols, rather than by acyclic graphs with vertices labeled by letters.  We  combine 
 the two approaches in order to relate them.

\begin{definition}
A 
{\bf symbol graph}
is an acyclic, connected, 
oriented graph whose vertices are labeled by symbols on a fixed generating set, so that if the vertices of two symbols are connected by an
edge then their free letters must be distinct.
Let $\gs$ denote the set of such and $\gs_{n,m}$ denote the subset with $m$ edges and whose symbols have depths which sum to $n$.
\end{definition}

The most important cases are  $m=0$, in which case we have  a symbol labeling a solitary vertex, and $n=0$ in which case
we have an acyclic graph whose vertices are generators such that edges only connect distinct generators.  
We call the latter {\bf distinct-vertex Eil graphs} because, as we prove below building on results of \cite{SiWa11},
they provide a model for the cofree Lie coalgebra on the generating set.  The intermediate cases with both $n,m \neq 0$ are needed
to relate symbols and Eil graphs, through a process used to define 
Hopf invariants in \cite{SiWa13}.

\begin{definition}
Let $v$ be a vertex in a symbol graph $G \in \gs_{n,m}$.  When all resulting terms are valid symbols, the 
{\bf reduction} of $G$ at $v$, denoted $\rho_v G$ is the linear
combination $\sum_{v \in \partial e} or_v(e)  G_{v,e} \in \Z \gs_{n+1,m-1}$ where 
\begin{itemize}
\item the sum is over edges $e$ incident upon $v$,
\item $or_v(e)$ is equal to $1$ if $e$ is oriented away from $v$ and $-1$ if oriented towards $v$,
 \item $G_{v,e}$ is obtained from $G$ by contracting the edge $e$ and labeling its image in the quotient by $(\sigma)\tau$ where $\sigma$
 is the label of $v$ and $\tau$ is the label of the other endpoint of $e$.
 \end{itemize}
 If any of the $G_{v,e}$ are not valid, we say the reduction of $G$ at $v$ is undefined.

  \end{definition}
 
  If there is a unique vertex $v$ labeled by some letter $\ell$, we may use the notation $\rho_\ell$ in place for $\rho_v$.
For example if $G = \graphpp{a}{b}{c}$, then 
$\rho_b (G) = -\linep{ a(b)}{c} + \linep{a}{ (b) c}$.  

This definition is motivated by the process of weight reduction 
in the Lie coalgebraic bar construction, used to define  Hopf invariants in  \cite{SiWa13}.  If one takes the definition of weight reduction
for Hopf invariants for higher homotopy groups of simply connected spaces and extends by applying it for the fundamental group
of a wedge of circles, one is led to this definition of reduction of symbol graphs.

\begin{definition}
If $w \neq v$ then by abuse we also use $w$ to denote the corresponding vertex under identification in any $G_{v,e}$.  By this convention,  
$\rho_w$ is  defined
on all such $G_{v,e}$, and we let the {\bf composite} $\rho_w \circ \rho_v (G)$ be defined by extending  linearly, namely 
$\sum_{v \in \partial e} or_v(e)  \rho_w G_{v,e}$, if all reductions are valid.

If $V = v_1, \ldots, v_k$ is a set of vertices of $G$ let $\rho_V$ be the composite $\rho_{v_k} \circ \rho_{v_{k-1}} \circ \cdots \circ  \rho_{v_1}$,
if defined.  If this composite is not defined we say $V$ is not valid.
\end{definition}
 
 Of primary interest is when one reduces to a sum of graphs each of which has a single vertex decorated by a symbol, 
 which we identify with the corresponding sum of symbols.
 For example $\rho_{b,a}$ of $G = \graphpp{a}{b}{c}$ is $ - (a(b)) c + ( a)(b) c$, while $\rho_{a,c} (G) = - (a) b (c)$.
 
\begin{definition}
Suppressing the set of generators from notation, let $\sy_n$ denote the set of {\bf symbols of depth $n$}, which is canonically identified with $\gs_{n,0}$.
Extend the letter-linking
homomorphisms $\Phi_{\sigma}$ linearly to $\Z \sy_n$, with the domain of definition of a linear combination of homomorphisms
given by intersection of the domains of the constituents.
\end{definition}
 
 We will find it fruitful to use not only symbols but also graphs to  parametrize letter-linking homomorphisms, 
as facilitated by the following main result. 
 
\begin{theorem}\label{graph-symbol}
Let $G \in \gs_{n,m}$.  The letter-linking 
homomorphism $\Phi_{{\rho_V} G}$ is independent of choice of valid ordered set of $m-1$ vertices $V = v_1, \ldots, v_{m-1}$. 
\end{theorem}
 
 In light of this theorem we shorten $\Phi_{\rho_V G}$ to just $\Phi_G$.
 
 In our example considering two different reductions of $G =  \graphpp{a}{b}{c}$ above, 
 this theorem says that $- \Phi_{(a(b)) c} + \Phi_{( a)(b) c} = - \Phi_{(a) b (c)}$,  which follows from the antisymmetry
 and Arnold identities.  Reduction of graphs through different vertex orders 
 thus gives a way of producing relations between letter-linking invariants.
  
 The following combinatorics is at the heart of the proof of Theorem~\ref{graph-symbol}
 
\begin{definition} \label{interleavingdef}
Recall that a {\bf cycle} of length $n$ in a set $S$ is an element of the quotient of $S^{\times n}$ by the cyclic group of order $n$.

Let $S$ be a signed,  partially ordered set with disjoint subsets $A$ and $B$.  
We say a cycle in $S$ is 
\begin{itemize}
\item  {\bf alternating} if consecutive terms have opposite  signs, 
\item  {\bf interleaving} (of $A$ and $B$) if it alternates between two elements of $A$, followed by two elements of $B$, etc.
\end{itemize}

We say a cycle {\bf crosses over} some element $c$, which is not in $A$ or $B$ and is ordered with respect to all of their elements, 
whenever $c$ occurs between consecutive terms in the sequence.

A crossing over is {\bf homogeneous} if the consecutive terms are both in $A$ or both in $B$, or {\bf heterogeneous} otherwise.
The {\bf sign} of such a crossing over is the sign of the term in the sequence which is less than $c$ in the total ordering
(irregardless of whether that term was earlier or later in the sequence).
\end{definition}

For example, let $S$ be the set of integers with standard order, with $A$ being the  integers less than $5$ and $B$ those greater than $5$, 
and  sign function which assigns $-1$ to
odd numbers   and $+1$ to even numbers, so being alternating means to alternate between even and odd.    The cycle $3 \to 4 \to  7 \to 8  (\to 3)$
crosses over $c=5$ twice, from $4 \to 7$ which has sign $+1$ since $4$ is even, and from $8 \to 3$, which has sign $-1$ since $3$ is odd.
Both of these crossings are heterogeneous, as are any which can occur for this $S, A, B$ and $c$.

\begin{lemma} \label{combolemma}
Let $S$, $A$ and $B$ be as above.  For any alternating, interleaving cycle  
 and any $c \notin A,B$ which is ordered with respect to them, the signed count of
heterogenous crossings over $c$ and that of homogeneous crossings are equal. 
\end{lemma}

\begin{proof}
Fix a sequence representative for the cycle $\{ s_j \}$, and suppose the sequence first crosses over $c$ between $s_i$
and $s_{i+1}$ which are both in $A$ or both in $B$, contributing $\epsilon = \pm 1$ to the count of homogeneous crossings.  
Let the next crossing over $c$ be in $k$ steps (that is, between $s_{i+k}$ and $s_{i+k+1}$).  Since the sequence
is alternating and interleaving, consider $k$ modulo four.  
\begin{itemize}
\item If $k = 1 $ or $3 \mod 4$ the next crossing  contributes $\epsilon$ to the count of heterogeneous crossings.
\item If $k = 0$ or $2 \mod 4$ the next crossing  contributes $- \epsilon$ to the count of homogenous crossings.
\end{itemize}  
In all cases the signed count of
heterogenous crossings and that of homogeneous crossings are equal.

If the first crossing is heterogenous, the argument is the same, with the roles of heterogeneous and homogeneous interchanged.
\end{proof}

We apply this combinatorics to lists which arise in different reductions of symbol graphs.  The general equivalence needed is the following. 

\begin{lemma} \label{threelistlemma}
Let $L_a$, $L_b$ and $L_c$ be  lists of distinct letters $a,b,c$ in a word $w$
with $d^{-1} (L_a \wedge d^{-1} L_b)$ and $d^{-1} (d^{-1} L_a \wedge  L_b)$ defined.
 Then $L_c \wedge d^{-1} L_a \wedge d^{-1}L_b$ is simply equivalent to the union of  
$L_c \wedge d^{-1} (L_a \wedge d^{-1} L_b)$ and  $L_c \wedge d^{-1} (d^{-1} L_a \wedge  L_b)$.
\end{lemma}

At the level of counts this follows from antisymmetry and Arnold identities, 
but we  need  simple equivalence to have equality of further derived  counts.

\begin{proof}
We  fix choices of $d^{-1} L_a$ and $d^{-1}L_b$.  Let  
\begin{itemize}
\item $B_{(a)(b)}$ denote the union of intersections of intervals in
$d^{-1} L_a$ and $d^{-1}L_b$,
\item $B_{((a)b)}$ be alternate notation for $d^{-1} (d^{-1} L_a \wedge  L_b)$, and
\item $B_{(a(b))}$ be alternate notation for $d^{-1}  (L_a \wedge d^{-1} L_b)$.
\end{itemize}
Our lemma states that the intersection of $L_c$ with $B_{(a)(b)}$ is simply equivalent to its intersection with $B_{((a)b)} \cup B_{(a(b))}$.
A key  observation is that the boundaries of the 
intervals in $B_{(a)(b)}$ and  those in $B_{((a)b)} \cup B_{(a(b))}$ coincide.

We first quickly address the case of containment of intervals in the $L_a$ and $L_b$ coboundings.
If some $I \in d^{-1} L_a$  is contained in some $J \in d^{-1} L_b$ then $I$ can be chosen
in $B_{(a(b))}$, 
and since $I \cap J = I$ it occurs in $B_{(a)(b)}$.  
So any intersections of $L_c$
with $I$ would be added equally for the two sets named.  

We thus focus on intersections of interleaving intervals from $d^{-1} L_a$ and $d^{-1} L_b$, which thus have one boundary point in $L_a$
and one in $L_b$, of opposite total signs.  With an eye to  applying
Lemma~\ref{combolemma}, set $A$ to be the collection with multiplicity
of all the elements of $L_a$  which are  boundaries of interleaving intervals from 
$d^{-1} L_a$ and $d^{-1} L_b$.  Let  $B$ the elements of $L_b$  
which are such 
boundaries, and $S$ be their union along with $L_c$.  Order  $S$ using the ordering of letters of $w$.  

By construction, points in $A$ and $B$ are all the boundaries of one interval from $B_{(a)(b)}$. 
They are also
the boundary of an interval in $B_{(a(b))}$ or $B_{((a)b)}$.  
Thus the union of interleaving
intervals from $B_{(a)(b)}$ and $B_{(a(b))} \cup B_{((a)b)}$ naturally define  cycles, with each boundary point connected to two edges, 
and each edge proceeding from one boundary point to the next.  
Starting with any point, following a cycle
will define an alternating, interleaving sequence.  
The heterogeneous crossings of a point in $L_c$ with this sequence are exactly contributions to  
$L_c \wedge B_{(a)(b)}$ 
while the homogeneous crossings are contributions to $L_c \wedge (B_{((a)b)} \cup B_{(a(b))})$.
By  Lemma~\ref{combolemma} these are equal,
from which we deduce this lemma.
\end{proof}

We can now prove that reduction of symbol graphs to symbols gives well-defined letter-linking invariants.

\begin{proof}[Proof of Theorem~\ref{graph-symbol}]
Let $V = v_1, \ldots, v_{m-1}, v_m$ be a list of $m-1$ vertices at which a symbol graph in $\gs_{n,m}$ is to be reduced, followed by the remaining 
vertex $v_m$.  Any two such lists differ by a sequence of transpositions, so it suffices to consider a $V'$ which differs by a single transposition.
Because the lists are the same up until the transposition, and thus will produce the same reductions up until that point,
 it suffices to consider a transposition of $v_1$ and $v_{2}$.  

If $v_1$ and $v_{2}$ are not connected by an edge then the resulting reductions will be the same, so we assume
there is an edge $e$ between them, oriented say away from $v_1$ towards $v_2$.  
Let $\sigma$ and $\tau$ be the symbols at $v_1$ and $v_2$ respectively.   If this edge is the last one in the graph, then the resulting reductions are
equivalent by antisymmetry, so we consider the other cases when reduction occurs at both vertices.
Each term in the linear combination of the reduction of $G$ at $v_1$ and then $v_2$ correspond to a choice of edge incident to $v_1$ and an
edge incident to $v_2$ in the quotient.  If neither of these edges is $e$ this term will be the same as the 
corresponding term in the reduction at $v_2$ and then $v_1$.    

Thus we consider reduction at $e$ along with 
a second edge $f$ incident to $v_1$, say oriented away from $v_1$, connected to some vertex $w$ labeled by symbol $\mu$.  There are
two terms in the reduction at $v_1$ and then $v_2$ which correspond to contraction of $e$ and $f$, namely $e$ could come first and then $f$,
resulting in the labeling symbol $\mu ((\sigma) \tau)$ at the resulting vertex in the quotient, 
or $f$ could come first and then $e$, resulting in $-\mu (\sigma) (\tau)$.  There is one term in the
reduction at $v_2$ and then $v_1$ as  $e$ must first be contracted then $f$, giving a labeling symbol $\mu (\sigma(\tau))$.

As the reduction of these terms will be identical after these contractions,
it suffices to know for any $w$ the union of the lists $\Lambda_{\mu ((\sigma) \tau)}(w)$ and $\Lambda_{\mu (\sigma) (\tau)}(w)$ 
with its orientations reversed is simply equivalent
to $\Lambda_{\mu (\sigma(\tau))}(w)$.  Using the group structure on simple equivalence classes of lists, it suffices
to have   $\Lambda_{\mu (\sigma) (\tau)}(w)$ simply equivalent to  
$\Lambda_{\mu ((\sigma) \tau)}(w) \cup \Lambda_{\mu (\sigma(\tau))}(w)$.  
But this is the content Lemma~\ref{threelistlemma}, setting $L_c =\Lambda_\mu (w)$,
$L_a = \Lambda_\sigma(w)$ and $L_b = \Lambda_\tau(w)$.
\end{proof}

To make full use of this reduction, we need the following simple  combinatorial fact.

\begin{proposition}\label{reductionsurjective}
For any symbol $\sigma$ there is a graph $G$ and sequence of vertices $V$ such that $\rho_V G = \sigma$.
\end{proposition}

\begin{proof}
One such graph $G$ essentially encodes the containment poset, with a vertex for every pair of parentheses along with a vertex for the entire symbol.  Each  vertex is
labeled by   the free letter for the corresponding pair of parentheses, or respectively the free letter for the symbol.  There is an edge from the vertex
of a set of parentheses to the set of parentheses which immediately contains it, or respectively to the vertex for the entire symbol for the parentheses
not contained in any others.  By reducing at any list  of vertices whose order is compatible with  the containment order of parentheses,
we obtain $\sigma$ as the reduction.
\end{proof}

\subsection{The configuration pairing}

The  lower central series filtration of a group is universal among filtrations whose subquotients form a Lie algebra.  In the case
of free groups, the resulting Lie algebra is free.  In \cite{SiWa11} the second author and Ben Walter developed an approach to free Lie 
algebras and their linear duals, starting with an operadic perspective.  We will connect with this approach to apply those results.  

\begin{definition}
Fix a set $x_1, \ldots, x_n$ of generators of $F_n$.  
Let $UC(n)$ denote the set of commutators in which each generator occurs a unique time, and 
let $\Lie(n)$ denote the submodule of $\gamma_n F_n/ \gamma_{n+1} F_n$ generated by $UC(n)$.
\end{definition}

 Combinatorially, these can be represented by trees.
 
 \begin{definition}
 Let $\tr(n)$ denote the set of isotopy classes of rooted,
half-planar, uni-trivalent trees with leaves labeled by integers $1\ldots n$.  

Represent a commutator  $w \in UC(n)$ by an element  $T_w \in \tr(n)$, starting with a one-edge tree with leaf label $i$ as  $T_{x_i}$.  
Then define $T_{[w,v]}$ to the tree formed by taking 
$T_w$ and $T_v$ and grafting them to a single (rooted) trivalent vertex, with $T_w$ on the left.
\end{definition}

Then $\Lie(n)$ is isomorphic to the quotient of $\Z \tr(n)$ by linear combinations corresponding to antisymmetry and Jacobi identities.

\begin{definition}
Let $\gr(n)$ denote the subset of $\gs_{0,n}$ given by acyclic oriented graphs in which each generator occurs exactly once.
\end{definition}

Thus $\gr(n)$ is the set of Eil graphs on $n$ vertices whose vertices are labeled by generators, which we indicate by using label $i$ in place of $x_i$.

We now develop the pairing between $\gr(n)$ and $\tr(n)$ first 
developed in \cite{Sinh18, Sinh09.2}, arising in the study of configuration spaces, which descend to a perfect
pairing between the Lie operad and Eil co-operad.

\begin{definition}\label{D:confpair}

Let the height of a vertex in a tree be the number of edges between that vertex
and the root, and let $gcv(i,j)$ be the vertex of greatest height which lies beneath
leaves labelled $i$ and $j$.

Given $G \in \gr(n)$ and $T \in \tr(n)$, define the map 
$$\beta_{G,T}:\bigl\{\text{edges of } G\bigr\} \longrightarrow 
\bigl\{\text{internal vertices of } T\bigr\}$$ 
by sending the edge $\linep{i}{j}$ in $G$ to the vertex $gcv(i,j)$ in $T$. 
The configuration pairing of
$G$ and $T$ is 
$$\bigl\langle G,\, T\bigr\rangle = 
  \begin{cases} \displaystyle
    \prod_{\substack{e\text{ an edge} \\\text{of }G}} \!\!\!\! 
        \text{sgn}\bigl(\beta_{G,T}(e)\bigr)
                & \text{if $\beta$ is surjective,} \\
         \qquad 0  & \text{otherwise},
\end{cases}$$
where   
$\text{sgn}\left(
  \beta_{G,T}\Bigl(\begin{aligned}\linep{i}{j}\end{aligned}\Bigr)\right) = 1$ 
   if leaf $i$
is to the left of leaf $j$ in the half-planar embedding of $T$;  
otherwise it is $-1$.  
\end{definition}

The first main result of this section is that this pairing also governs our letter-linking invariants in the case
where the symbols and commutators have each generator occurring only once.

\begin{theorem}\label{graph-symbol-pairing}
Let $G \in \gr(n)$ and $w \in UC(n)$.
Then $\Phi_G(w) =  \langle G, T_w \rangle$.
\end{theorem}

\begin{proof}
If $\langle G, T_w \rangle = \pm 1$ we argue inductively, at first not tracking signs.  
The theorem is immediate when $n=1$.  If $n > 1$, set $w = [w_1, w_2]$.
As $\langle G, T_w \rangle = \pm 1$ then there is a unique edge $e$ of $G$ such that $\beta_{G, T_w}(e) = v$.   Removing $e$ from $G$ 
yields two acyclic graphs $G_1$ and $G_2$, whose vertex labels must coincide with those of $T_1$ and $T_2$ respectively, since any other
edge connecting a vertex with label among those in $T_1$ to a vertex with label from $T_2$ would also have its image under $\beta_{G, T_w}$
equal to $v$.  Moreover, we must have $\langle G_i, T_i \rangle = \pm 1$ for $i = 1,2$, which inductively implies $\Phi_{G_i} (w_i) = \pm 1$.  

Now choose to reduce $G$ so that the vertices of $e$ are the last two and then, say, choose
the vertex in $G_1$ for the last reduction.  In this case, the symbol for $\rho_V G$
will be $(\sigma_1) \sigma_2$, where $\sigma_i$ is the symbol reduction of $G_i$.  Because the generators which occur in $w_1$ and $w_2$
are distinct, $\Lambda_{\sigma_1} (w_1 w_2 {w_1}^{-1} {w_2}^{-1})$ occurs only in the $w_1$ and ${w_1}^{-1}$ sub-words.  The occurrences
of the free letter of $\sigma_1$ will occur in pairs  across $w_1$ and ${w_1}^{-1}$ mapped to each other by the canonical involution,
with a total multiplicity of such pairs of $\pm 1$.  
We cobound according to this choice of pairs.  Similarly $\Lambda_{\sigma_2}(w)$
will only occur in $w_2$ and ${w_2}^{-1}$ sub-words.  Only the occurrences in $w_2$ will intersect the cobounding, and by the inductive
assumption that $\Phi_{G_i} (w_i) = \pm 1$ we have 
 $\Phi_{\pm(\sigma_1) \sigma_2} = \pm 1$.  

We obtain the signed result  by noticing that 
$ \langle G, T_w \rangle = \varepsilon \langle G_1, T_{w_1} \rangle  \langle G_2, T_{w_2} \rangle$,
where $\varepsilon = 1$ if the initial vertex of  $e$ is in $G_1$ or $-1$ if its initial vertex is in $G_2$.  Choose the last vertex for reduction to be the
initial vertex of $e$, so  that the result will be $(\sigma_1) \sigma_2$ if $e$ points from $G_1$ to $G_2$ or $\sigma_1 (\sigma_2)$ otherwise.  The first
case was chosen above, and we now have a signed equality $\Phi_G(w) = \Phi_{G_1} (w_1) \Phi_{G_2} (w_2)$.  In the second case,
it is the elements of $\Lambda_{\sigma_1} ({w_1}^{-1})$ which occur between pairs of $\Lambda_{\sigma_2} (w_2^{\pm 1})$ so we have 
$\Phi_G(w) = - \Phi_{G_1} (w_1) \Phi_{G_2} (w_2)$.  Both cases agree with the inductive formula for  $ \langle G, T_w \rangle$.

If $\langle G, T_w \rangle = 0$, there is a vertex $v$ with $\beta_{G, T_w}(e) = \beta_{G, T_w}(f) = v$ for at least two edges $e$ 
and $f$.  Reduce $G$ so that the images in the quotient of these two edges are the last two edges, 
with remaining symbols $\sigma$, $\tau$, $\mu$.  That is, reduce, up to edge 
orientation, to  $\graphpm{\sigma}{\tau}{\mu}$.  The final reduction can thus be chosen as $\pm (\sigma) \tau (\mu)$.  
Because $\beta_{G, T_w}(e) = \beta_{G, T_w}(f) = v$, the free letters of $\sigma$ and $\mu$ occur in $w$  within 
a commutator which is later commuted with the free letter of $\tau$.  The coboundings of free letters of $\sigma$ and $\mu$ can  
be chosen within this first commutator, and thus 
disjoint from the free letter of $\tau$, implying  $\Phi_G (w) = 0$.
\end{proof}

\subsection{Cofree Lie coalgebras and distinct-vertex graphs} 

The results  of the previous section addressing  symbols and words with distinct generators may be viewed as  ``at the level of operads.''  
We expand our consideration to the (co)free (co)algebras built from
them.  As needed for our work in group theory, we continue to build from a generating set rather than from a vector space.

\begin{definition} 
Let $\Eil(i)$ be the quotient of the span of $\gr(i)$ by antisymmetry and Arnold relations which are shown below. 

Let $\E_n$ be the quotient of the span of Eil graphs of any size with vertices  that are labelled by the generating set $x_1, \cdots, x_n$, 
by antisymmetry and Arnold relations
\begin{align*}
\text{(antisymmetry)}\qquad & \qquad
\begin{xy}                           
  (0,-2)*+UR{\scriptstyle a}="a",    
  (3,3)*+UR{\scriptstyle b}="b",     
  "a";"b"**\dir{-}?>*\dir{>},         
  (1.5,-5),{\ar@{. }@(l,l)(1.5,6)},
  ?!{"a";"a"+/va(210)/}="a1",
  ?!{"a";"a"+/va(240)/}="a2",
  ?!{"a";"a"+/va(270)/}="a3",
  "a";"a1"**\dir{-},  "a";"a2"**\dir{-},  "a";"a3"**\dir{-},
  (1.5,6),{\ar@{. }@(r,r)(1.5,-5)},
  ?!{"b";"b"+/va(90)/}="b1",
  ?!{"b";"b"+/va(30)/}="b2",
  ?!{"b";"b"+/va(60)/}="b3",
  "b";"b1"**\dir{-},  "b";"b2"**\dir{-},  "b";"b3"**\dir{-},
\end{xy}\ =\ \ -  
\begin{xy}                           
  (0,-2)*+UR{\scriptstyle a}="a",    
  (3,3)*+UR{\scriptstyle b}="b",     
  "a";"b"**\dir{-}?<*\dir{<},         
  (1.5,-5),{\ar@{. }@(l,l)(1.5,6)},
  ?!{"a";"a"+/va(210)/}="a1",
  ?!{"a";"a"+/va(240)/}="a2",
  ?!{"a";"a"+/va(270)/}="a3",
  "a";"a1"**\dir{-},  "a";"a2"**\dir{-},  "a";"a3"**\dir{-},
  (1.5,6),{\ar@{. }@(r,r)(1.5,-5)},
  ?!{"b";"b"+/va(90)/}="b1",
  ?!{"b";"b"+/va(30)/}="b2",
  ?!{"b";"b"+/va(60)/}="b3",
  "b";"b1"**\dir{-},  "b";"b2"**\dir{-},  "b";"b3"**\dir{-},
\end{xy} \\
\text{(Arnold)}\qquad & \qquad
\begin{xy}                           
  (0,-2)*+UR{\scriptstyle a}="a",    
  (3,3)*+UR{\scriptstyle b}="b",   
  (6,-2)*+UR{\scriptstyle c}="c",   
  "a";"b"**\dir{-}?>*\dir{>},         
  "b";"c"**\dir{-}?>*\dir{>},         
  (3,-5),{\ar@{. }@(l,l)(3,6)},
  ?!{"a";"a"+/va(210)/}="a1",
  ?!{"a";"a"+/va(240)/}="a2",
  ?!{"a";"a"+/va(270)/}="a3",
  ?!{"b";"b"+/va(120)/}="b1",
  "a";"a1"**\dir{-},  "a";"a2"**\dir{-},  "a";"a3"**\dir{-},
  "b";"b1"**\dir{-}, "b";(3,6)**\dir{-},
  (3,-5),{\ar@{. }@(r,r)(3,6)},
  ?!{"c";"c"+/va(-90)/}="c1",
  ?!{"c";"c"+/va(-60)/}="c2",
  ?!{"c";"c"+/va(-30)/}="c3",
  ?!{"b";"b"+/va(60)/}="b3",
  "c";"c1"**\dir{-},  "c";"c2"**\dir{-},  "c";"c3"**\dir{-},
  "b";"b3"**\dir{-}, 
\end{xy}\ + \                             
\begin{xy}                           
  (0,-2)*+UR{\scriptstyle a}="a",    
  (3,3)*+UR{\scriptstyle b}="b",   
  (6,-2)*+UR{\scriptstyle c}="c",    
  "b";"c"**\dir{-}?>*\dir{>},         
  "c";"a"**\dir{-}?>*\dir{>},          
  (3,-5),{\ar@{. }@(l,l)(3,6)},
  ?!{"a";"a"+/va(210)/}="a1",
  ?!{"a";"a"+/va(240)/}="a2",
  ?!{"a";"a"+/va(270)/}="a3",
  ?!{"b";"b"+/va(120)/}="b1",
  "a";"a1"**\dir{-},  "a";"a2"**\dir{-},  "a";"a3"**\dir{-},
  "b";"b1"**\dir{-}, "b";(3,6)**\dir{-},
  (3,-5),{\ar@{. }@(r,r)(3,6)},
  ?!{"c";"c"+/va(-90)/}="c1",
  ?!{"c";"c"+/va(-60)/}="c2",
  ?!{"c";"c"+/va(-30)/}="c3",
  ?!{"b";"b"+/va(60)/}="b3",
  "c";"c1"**\dir{-},  "c";"c2"**\dir{-},  "c";"c3"**\dir{-},
  "b";"b3"**\dir{-}, 
\end{xy}\ + \                              
\begin{xy}                           
  (0,-2)*+UR{\scriptstyle a}="a",    
  (3,3)*+UR{\scriptstyle b}="b",   
  (6,-2)*+UR{\scriptstyle c}="c",    
  "a";"b"**\dir{-}?>*\dir{>},         
  "c";"a"**\dir{-}?>*\dir{>},          
  (3,-5),{\ar@{. }@(l,l)(3,6)},
  ?!{"a";"a"+/va(210)/}="a1",
  ?!{"a";"a"+/va(240)/}="a2",
  ?!{"a";"a"+/va(270)/}="a3",
  ?!{"b";"b"+/va(120)/}="b1",
  "a";"a1"**\dir{-},  "a";"a2"**\dir{-},  "a";"a3"**\dir{-},
  "b";"b1"**\dir{-}, "b";(3,6)**\dir{-},
  (3,-5),{\ar@{. }@(r,r)(3,6)},
  ?!{"c";"c"+/va(-90)/}="c1",
  ?!{"c";"c"+/va(-60)/}="c2",
  ?!{"c";"c"+/va(-30)/}="c3",
  ?!{"b";"b"+/va(60)/}="b3",
  "c";"c1"**\dir{-},  "c";"c2"**\dir{-},  "c";"c3"**\dir{-},
  "b";"b3"**\dir{-}, 
\end{xy}\ =\ 0,                             
\end{align*}
along with setting graphs with cycles to zero.
\end{definition}

If we let $W$ be the span of $x_1, \cdots, x_n$, then $\E_n \cong \bigoplus_i \Eil(i) \otimes_{\sym_i} W^{\otimes i}$, where the 
symmetric group $\sym_i$ acts on $\Eil(i)$ by permuting vertex labels and on $W^{\otimes i}$ as usual by permuting
factors.  Recall that the free Lie algebra
$\fL_n$ is isomorphic to $\bigoplus_i \Lie(i) \otimes_{\sym_i}  W^{\otimes i}$. A key result of \cite{SiWa11}, namely
its Corollary 3.11, is the following.

\begin{theorem}\label{LieEilPairing}
The cofree Lie coalgebra $\E_n$ pairs perfectly with $\fL_n$ through the extension of the configuration pairing between all 
$\Eil(i)$ and $\Lie(i)$ and the
Kronecker pairing of $W$ with itself extended to $W^{\otimes i}$.
\end{theorem}

The graphs which define $\E_n$ include our symbol graphs $\gs_{0,n}$, but in general are not required to have distinct generators labeling
the endpoints of any edge.  

\begin{definition}
An edge in a labeled graph is called {\bf homogeneous} if it connects vertices with the same label and {\bf heterogeneous} otherwise.  
\end{definition}

Using this terminology, our previously defined distinct-vertex Eil graphs are those for whom all edges are heterogeneous.  
In \cite{SiWa11} it is noted that linear graphs -- that is,  connected, acyclic graphs whose vertices each have valence at most two -- span $\E_n$.
In \cite{ShWa16} a new basis is constructed using ``star graphs''.  A new spanning set is crucial for our present work.

\begin{theorem}\label{distinctspan}
The distinct vertex Eil graphs span $\E_n$ over the rational numbers.
\end{theorem}

\begin{proof}
We use  the defining antisymmetry and Arnold identities to express any 
graph in $\E_n$ as a linear combination of graphs with fewer homogeneous edges, yielding the result by 
induction.   
The reduction process is delicate since we will have recurrent appearances of terms, so care with signs is essential.

Consider any representative graph in $\E_n$. 
Pick any (co)generator $a$ of the cofree Lie coalgebra which appears in the graph and consider a maximal connected subgraph consisting entirely 
of edges connecting that generator to itself.  As is standard, see for example \cite{Sinh09.2, Sinh18}, the 
Arnold identity applied to that subgraph can be used to reduce valence and ultimately yield a linear combination of linear 
graphs.  Because we apply  identities exclusively to edges in the subgraph, these will have the same number of homogeneous edges.  

The resulting graphs will each have a linear homogeneous subgraph  all of whose vertices except those at the ends have two homogeneous edges.
All other edges incident to these vertices, which we call normal edges since they connect outside the subgraph, are heterogeneous.
 

 We show such a graph $G$ with a linear homogeneous subgraph can be reduced. 
 We first ``move'' all but one of the normal edges to one end of the linear graph 
 (in our pictures, to the right) using the Arnold identity as follows:\\

\begin{align*}
 &G  \,\,\,\, = \begin{xy}                           
  (12,5)*+UR{\scriptstyle c}="c1",    
  (9,0)*+UR{\scriptstyle a}="a1",   
  (15,0)*+UR{\scriptstyle a}="a2", 
    (3,0)*+UR{ }="o2",   
    (21,0)*+UR{ }="o3",    
  "o2";"a1"**\dir{-}?>*\dir{>},         
  "a1";"a2"**\dir{-}?>*\dir{>},       
   "a2";"o3"**\dir{-}?>*\dir{>}, 
   "c1";"a1"**\dir{-}?>*\dir{>},         
    (9,-4),{\ar@{.}@(l,l)(9,8)}, 
  ?!{"o2";"o2"+/va(0)/}="o21", 
  "o2";"o21"**\dir{.}, 
   (15,-4),{\ar@{.}@(r,r)(15,8)}, 
  ?!{"o3";"o3"+/va(0)/}="o31", 
  "o3";"o31"**\dir{.}, 
   (9,8),{\ar@{.}(15,8)}, 
   ?!{"c1";"c1"+/va(110)/}="c11", 
    ?!{"c1";"c1"+/va(90)/}="c12",
   ?!{"c1";"c1"+/va(70)/}="c13", 
      "c1";"c13"**\dir{-},  
      "c1";"c11"**\dir{-},
  "c1";"c12"**\dir{-}, 
    (9,-4),{\ar@{.}(15,-4)},   
\end{xy} = \,\,\,\, -
\begin{xy}                           
  (12,5)*+UR{\scriptstyle c}="c1",    
  (9,0)*+UR{\scriptstyle a}="a1",   
  (15,0)*+UR{\scriptstyle a}="a2", 
    (3,0)*+UR{ }="o2",   
    (21,0)*+UR{ }="o3",    
  "o2";"a1"**\dir{-}?>*\dir{>},         
  "a1";"a2"**\dir{-}?>*\dir{>},       
   "a2";"o3"**\dir{-}?>*\dir{>}, 
   "a2";"c1"**\dir{-}?>*\dir{>},         
    (9,-4),{\ar@{.}@(l,l)(9,8)}, 
  ?!{"o2";"o2"+/va(0)/}="o21", 
  "o2";"o21"**\dir{.}, 
   (15,-4),{\ar@{.}@(r,r)(15,8)}, 
  ?!{"o3";"o3"+/va(0)/}="o31", 
  "o3";"o31"**\dir{.}, 
   (9,8),{\ar@{.}(15,8)}, 
   ?!{"c1";"c1"+/va(110)/}="c11", 
    ?!{"c1";"c1"+/va(90)/}="c12",
   ?!{"c1";"c1"+/va(70)/}="c13", 
      "c1";"c13"**\dir{-},  
      "c1";"c11"**\dir{-},
  "c1";"c12"**\dir{-}, 
    (9,-4),{\ar@{.}(15,-4)},
 \end{xy}-
\begin{xy}                           
  (12,5)*+UR{\scriptstyle c}="c1",    
  (9,0)*+UR{\scriptstyle a}="a1",   
  (15,0)*+UR{\scriptstyle a}="a2", 
    (3,0)*+UR{ }="o2",   
    (21,0)*+UR{ }="o3",    
  "o2";"a1"**\dir{-}?>*\dir{>},         
  "a2";"c1"**\dir{-}?>*\dir{>},       
   "a2";"o3"**\dir{-}?>*\dir{>}, 
   "c1";"a1"**\dir{-}?>*\dir{>},         
    (9,-4),{\ar@{.}@(l,l)(9,8)}, 
  ?!{"o2";"o2"+/va(0)/}="o21", 
  "o2";"o21"**\dir{.}, 
   (15,-4),{\ar@{.}@(r,r)(15,8)}, 
  ?!{"o3";"o3"+/va(0)/}="o31", 
  "o3";"o31"**\dir{.}, 
   (9,8),{\ar@{.}(15,8)}, 
   ?!{"c1";"c1"+/va(110)/}="c11", 
    ?!{"c1";"c1"+/va(90)/}="c12",
   ?!{"c1";"c1"+/va(70)/}="c13", 
      "c1";"c13"**\dir{-},  
      "c1";"c11"**\dir{-},
  "c1";"c12"**\dir{-}, 
    (9,-4),{\ar@{.}(15,-4)},
    \end{xy}.
\end{align*}
Note that there may be other edges connected to these vertices that are not drawn in the picture.  
No homogenous edges have been added in any graph.  In the middle graph, $c$ is connected to an $a$ which is one further to the right, as desired.
In the third graph we have decreased the number of homogeneous edges by one.

Reserving $=$ for equality in $\E_n$ (that is, modulo Arnold and antisymmetry relations), 
we let $\sim$  denote equivalence in $\E_n$ modulo graphs with fewer homogeneous edges, which are reducible by induction hypothesis.
Applying the above identity repeatedly, we have
\begin{align*}
&G \,\,\,\, \sim 
\begin{xy}                           
  (0,0)*+UR{\scriptstyle b}="b1",    
  (6,0)*+UR{\scriptstyle a}="a1",   
  (12,0)*+UR{\scriptstyle a}="a2", 
  (18,0)*+UR{\scriptstyle a}="a3",  
  (24,0)*+UR{ }="o1",    
    (27,0)*+UR{ }="o2",    
  (33,0)*+UR{\scriptstyle a}="a4",       
  "b1";"a1"**\dir{-}?>*\dir{>},         
  "a1";"a2"**\dir{-}?>*\dir{>},         
  "a2";"a3"**\dir{-}?>*\dir{>},       
   "a3";"o1"**\dir{-}?>*\dir{>}, 
   "o1";"o2"**\dir{.}?>*\dir{.},    
   "o2";"a4"**\dir{-}?>*\dir{>},                                      
  (6,-7),{\ar@{.}@(l,l)(6,7)}, 
  ?!{"b1";"b1"+/va(0)/}="b11", 
  ?!{"b1";"b1"+/va(30)/}="b12",
  ?!{"b1";"b1"+/va(-30)/}="b13",
  "b1";"b11"**\dir{-},  "b1";"b12"**\dir{-},  "b1";"b13"**\dir{-},
   (27,-7),{\ar@{.}@(r,r)(27,7)},
   ?!{"a4";"a4"+/va(0)/}="a41", 
  ?!{"a4";"a4"+/va(30)/}="a42",
  ?!{"a4";"a4"+/va(-30)/}="a43",
  "a4";"a41"**\dir{-},  "a4";"a42"**\dir{-},  "a4";"a43"**\dir{-},
    (6,7),{\ar@{.}(27,7)},
    (6,-7),{\ar@{.}(27,-7)},
\end{xy},
\end{align*}
where there could be  single normal edge on the left-most $a$, multiple possible on far right $a$, and no normal edges
on the ``middle'' $a$'s. 

We call the graph on the right-hand side $G_0$ and now reduce it.  
Apply the Arnold identity to the $b \to a \to a$ subgraph to get 

\begin{align*}
& G_0 \,\,\,\, = \,\,\,\, -
\begin{xy}                           
  (0,0)*+UR{\scriptstyle b}="b1",    
  (3,5)*+UR{\scriptstyle a}="a1",   
  (6,0)*+UR{\scriptstyle a}="a2", 
  (12,0)*+UR{\scriptstyle a}="a3", 
  (18,0)*+UR{ }="o1",    
    (21,0)*+UR{ }="o2",    
  (27,0)*+UR{\scriptstyle a}="a4",       
  "a2";"b1"**\dir{-}?>*\dir{>},         
  "a1";"a2"**\dir{-}?>*\dir{>},         
  "a2";"a3"**\dir{-}?>*\dir{>},   
    "a3";"o1"**\dir{-}?>*\dir{>},           
   "o1";"o2"**\dir{.}?>*\dir{.},    
   "o2";"a4"**\dir{-}?>*\dir{>},                                      
  (6,-7),{\ar@{.}@(l,l)(6,7)}, 
  ?!{"b1";"b1"+/va(0)/}="b11", 
  ?!{"b1";"b1"+/va(30)/}="b12",
  ?!{"b1";"b1"+/va(-30)/}="b13",
  "b1";"b11"**\dir{-},  "b1";"b12"**\dir{-},  "b1";"b13"**\dir{-},
   (21,-7),{\ar@{.}@(r,r)(21,7)},
   ?!{"a4";"a4"+/va(0)/}="a41", 
  ?!{"a4";"a4"+/va(30)/}="a42",
  ?!{"a4";"a4"+/va(-30)/}="a43",
  "a4";"a41"**\dir{-},  "a4";"a42"**\dir{-},  "a4";"a43"**\dir{-},
    (6,7),{\ar@{.}(21,7)},
    (6,-7),{\ar@{.}(21,-7)},
\end{xy} -
\begin{xy}                           
  (0,0)*+UR{\scriptstyle b}="b1",    
  (3,5)*+UR{\scriptstyle a}="a1",   
  (6,0)*+UR{\scriptstyle a}="a2", 
  (12,0)*+UR{\scriptstyle a}="a3", 
  (18,0)*+UR{ }="o1",    
    (21,0)*+UR{ }="o2",    
  (27,0)*+UR{\scriptstyle a}="a4",       
  "b1";"a1"**\dir{-}?>*\dir{>},         
  "a2";"b1"**\dir{-}?>*\dir{>},         
  "a2";"a3"**\dir{-}?>*\dir{>},   
    "a3";"o1"**\dir{-}?>*\dir{>},           
   "o1";"o2"**\dir{.}?>*\dir{.},    
   "o2";"a4"**\dir{-}?>*\dir{>},                                      
  (6,-7),{\ar@{.}@(l,l)(6,7)}, 
  ?!{"b1";"b1"+/va(0)/}="b11", 
  ?!{"b1";"b1"+/va(30)/}="b12",
  ?!{"b1";"b1"+/va(-30)/}="b13",
  "b1";"b11"**\dir{-},  "b1";"b12"**\dir{-},  "b1";"b13"**\dir{-},
   (21,-7),{\ar@{.}@(r,r)(21,7)},
   ?!{"a4";"a4"+/va(0)/}="a41", 
  ?!{"a4";"a4"+/va(30)/}="a42",
  ?!{"a4";"a4"+/va(-30)/}="a43",
  "a4";"a41"**\dir{-},  "a4";"a42"**\dir{-},  "a4";"a43"**\dir{-},
    (6,7),{\ar@{.}(21,7)},
    (6,-7),{\ar@{.}(21,-7)},
\end{xy}.
\end{align*}

Using antisymmetry on the heterogeneous edge connecting $b$ and $a$ we have
\begin{align*}
 & G_0\,\,\,\, = 
\begin{xy}                           
  (0,0)*+UR{\scriptstyle b}="b1",    
  (3,5)*+UR{\scriptstyle a}="a1",   
  (6,0)*+UR{\scriptstyle a}="a2", 
  (12,0)*+UR{\scriptstyle a}="a3", 
  (18,0)*+UR{ }="o1",    
    (21,0)*+UR{ }="o2",    
  (27,0)*+UR{\scriptstyle a}="a4",       
  "b1";"a2"**\dir{-}?>*\dir{>},         
  "a1";"a2"**\dir{-}?>*\dir{>},         
  "a2";"a3"**\dir{-}?>*\dir{>},   
    "a3";"o1"**\dir{-}?>*\dir{>},           
   "o1";"o2"**\dir{.}?>*\dir{.},    
   "o2";"a4"**\dir{-}?>*\dir{>},                                      
  (6,-7),{\ar@{.}@(l,l)(6,7)}, 
  ?!{"b1";"b1"+/va(0)/}="b11", 
  ?!{"b1";"b1"+/va(30)/}="b12",
  ?!{"b1";"b1"+/va(-30)/}="b13",
  "b1";"b11"**\dir{-},  "b1";"b12"**\dir{-},  "b1";"b13"**\dir{-},
   (21,-7),{\ar@{.}@(r,r)(21,7)},
   ?!{"a4";"a4"+/va(0)/}="a41", 
  ?!{"a4";"a4"+/va(30)/}="a42",
  ?!{"a4";"a4"+/va(-30)/}="a43",
  "a4";"a41"**\dir{-},  "a4";"a42"**\dir{-},  "a4";"a43"**\dir{-},
    (6,7),{\ar@{.}(21,7)},
    (6,-7),{\ar@{.}(21,-7)},
\end{xy} -
\begin{xy}                           
  (0,0)*+UR{\scriptstyle b}="b1",    
  (3,5)*+UR{\scriptstyle a}="a1",   
  (6,0)*+UR{\scriptstyle a}="a2", 
  (12,0)*+UR{\scriptstyle a}="a3", 
  (18,0)*+UR{ }="o1",    
    (21,0)*+UR{ }="o2",    
  (27,0)*+UR{\scriptstyle a}="a4",       
  "b1";"a1"**\dir{-}?>*\dir{>},         
  "a2";"b1"**\dir{-}?>*\dir{>},         
  "a2";"a3"**\dir{-}?>*\dir{>},   
    "a3";"o1"**\dir{-}?>*\dir{>},           
   "o1";"o2"**\dir{.}?>*\dir{.},    
   "o2";"a4"**\dir{-}?>*\dir{>},                                      
  (6,-7),{\ar@{.}@(l,l)(6,7)}, 
  ?!{"b1";"b1"+/va(0)/}="b11", 
  ?!{"b1";"b1"+/va(30)/}="b12",
  ?!{"b1";"b1"+/va(-30)/}="b13",
  "b1";"b11"**\dir{-},  "b1";"b12"**\dir{-},  "b1";"b13"**\dir{-},
   (21,-7),{\ar@{.}@(r,r)(21,7)},
   ?!{"a4";"a4"+/va(0)/}="a41", 
  ?!{"a4";"a4"+/va(30)/}="a42",
  ?!{"a4";"a4"+/va(-30)/}="a43",
  "a4";"a41"**\dir{-},  "a4";"a42"**\dir{-},  "a4";"a43"**\dir{-},
    (6,7),{\ar@{.}(21,7)},
    (6,-7),{\ar@{.}(21,-7)},
\end{xy}.
\end{align*}

Let $G_1$ be the first graph on the right hand side.
The second term on the right hand side has fewer homogeneous edges, so $G_0 \sim G_1$.  

We now  ``transfer the normal $a \to $ in $G_1$ down the graph to the end.'' 
We first apply the Arnold identity to this normal edge and the one following it in the linear chain, and then
antisymmetry to the resulting terms to obtain

\begin{align*}
& G_1\,\,\,\, = 
\begin{xy}                           
  (0,0)*+UR{\scriptstyle b}="b1",    
  (9,5)*+UR{\scriptstyle a}="a1",   
  (6,0)*+UR{\scriptstyle a}="a2", 
  (12,0)*+UR{\scriptstyle a}="a3", 
  (18,0)*+UR{ }="o1",    
    (21,0)*+UR{ }="o2",    
  (27,0)*+UR{\scriptstyle a}="a4",       
  "b1";"a2"**\dir{-}?>*\dir{>},         
  "a1";"a3"**\dir{-}?>*\dir{>},         
  "a2";"a3"**\dir{-}?>*\dir{>},   
    "a3";"o1"**\dir{-}?>*\dir{>},           
   "o1";"o2"**\dir{.}?>*\dir{.},    
   "o2";"a4"**\dir{-}?>*\dir{>},                                      
  (6,-7),{\ar@{.}@(l,l)(6,7)}, 
  ?!{"b1";"b1"+/va(0)/}="b11", 
  ?!{"b1";"b1"+/va(30)/}="b12",
  ?!{"b1";"b1"+/va(-30)/}="b13",
  "b1";"b11"**\dir{-},  "b1";"b12"**\dir{-},  "b1";"b13"**\dir{-},
   (21,-7),{\ar@{.}@(r,r)(21,7)},
   ?!{"a4";"a4"+/va(0)/}="a41", 
  ?!{"a4";"a4"+/va(30)/}="a42",
  ?!{"a4";"a4"+/va(-30)/}="a43",
  "a4";"a41"**\dir{-},  "a4";"a42"**\dir{-},  "a4";"a43"**\dir{-},
    (6,7),{\ar@{.}(21,7)},
    (6,-7),{\ar@{.}(21,-7)},
\end{xy} -
\begin{xy}                           
  (0,0)*+UR{\scriptstyle b}="b1",    
  (9,5)*+UR{\scriptstyle a}="a1",   
  (6,0)*+UR{\scriptstyle a}="a2", 
  (12,0)*+UR{\scriptstyle a}="a3", 
  (18,0)*+UR{ }="o1",    
    (21,0)*+UR{ }="o2",    
  (27,0)*+UR{\scriptstyle a}="a4",       
  "a1";"a3"**\dir{-}?>*\dir{>},         
  "b1";"a2"**\dir{-}?>*\dir{>},         
  "a2";"a1"**\dir{-}?>*\dir{>},   
    "a3";"o1"**\dir{-}?>*\dir{>},           
   "o1";"o2"**\dir{.}?>*\dir{.},    
   "o2";"a4"**\dir{-}?>*\dir{>},                                      
  (6,-7),{\ar@{.}@(l,l)(6,7)}, 
  ?!{"b1";"b1"+/va(0)/}="b11", 
  ?!{"b1";"b1"+/va(30)/}="b12",
  ?!{"b1";"b1"+/va(-30)/}="b13",
  "b1";"b11"**\dir{-},  "b1";"b12"**\dir{-},  "b1";"b13"**\dir{-},
   (21,-7),{\ar@{.}@(r,r)(21,7)},
   ?!{"a4";"a4"+/va(0)/}="a41", 
  ?!{"a4";"a4"+/va(30)/}="a42",
  ?!{"a4";"a4"+/va(-30)/}="a43",
  "a4";"a41"**\dir{-},  "a4";"a42"**\dir{-},  "a4";"a43"**\dir{-},
    (6,7),{\ar@{.}(21,7)},
    (6,-7),{\ar@{.}(21,-7)},
\end{xy}.
\end{align*}

Call the first graph on the right hand side $G_2$, and  notice the last graph is $G_0$.  
So $G_1=G_2-G_0$, and thus $2 G_0 \sim G_2$.

In general, let 
\begin{align*}
G_i\,\,\,\, = 
\begin{xy}                           
  (0,0)*+UR{\scriptstyle b}="b1",    
  (6,0)*+UR{\scriptstyle a}="a1",   
    (12,0)*+UR{ }="o1",    
        (15,0)*+UR{ }="o2",    
  (21,0)*+UR{\scriptstyle a}="a2", 
    (24,5)*+UR{\scriptstyle a}="ai",   
  (27,0)*+UR{\scriptstyle a}="a3", 
  (33,0)*+UR{ }="o3",    
    (36,0)*+UR{ }="o4",    
  (42,0)*+UR{\scriptstyle a}="a4",       
  "b1";"a1"**\dir{-}?>*\dir{>},         
  "a1";"o1"**\dir{-}?>*\dir{>}, 
      "o1";"o2"**\dir{.}?>*\dir{.},                         
  "o2";"a2"**\dir{-}?>*\dir{>},   
    "a2";"a3"**\dir{-}?>*\dir{>},   
      "ai";"a2"**\dir{-}?>*\dir{>},   
    "a3";"o3"**\dir{-}?>*\dir{>},           
   "o3";"o4"**\dir{.}?>*\dir{.},    
   "o4";"a4"**\dir{-}?>*\dir{>},                                      
  (6,-7),{\ar@{.}@(l,l)(6,7)}, 
  ?!{"b1";"b1"+/va(0)/}="b11", 
  ?!{"b1";"b1"+/va(30)/}="b12",
  ?!{"b1";"b1"+/va(-30)/}="b13",
  "b1";"b11"**\dir{-},  "b1";"b12"**\dir{-},  "b1";"b13"**\dir{-},
   (36,-7),{\ar@{.}@(r,r)(36,7)},
   ?!{"a4";"a4"+/va(0)/}="a41", 
  ?!{"a4";"a4"+/va(30)/}="a42",
  ?!{"a4";"a4"+/va(-30)/}="a43",
  "a4";"a41"**\dir{-},  "a4";"a42"**\dir{-},  "a4";"a43"**\dir{-},
    (6,7),{\ar@{.}(36,7)},
    (6,-7),{\ar@{.}(36,-7)},
\end{xy},
\end{align*}
where the ``normal'' $a \to$ is connected to the $i$th $a$ in the linear chain.
We argue as above, applying the Arnold identity to the normal edge and edge which follows it to deduce $G_i=G_{i+1}-G_0$.  
Hence for each $i<n$, where $n$ is the length of the linear chain of $a$'s,
\[
G_1=G_2-G_0=G_3-2G_0=\cdots = G_i-(i-1)G_0.
\]
 
 We thus focus on 
\begin{align*}
G_n\,\,\,\, = 
\begin{xy}                           
  (0,0)*+UR{\scriptstyle b}="b1",    
  (6,0)*+UR{\scriptstyle a}="a1",   
    (12,0)*+UR{ }="o1",    
        (15,0)*+UR{ }="o2",    
  (21,0)*+UR{\scriptstyle a}="a2", 
    (39,5)*+UR{\scriptstyle a}="ai",   
  (27,0)*+UR{\scriptstyle a}="a3", 
  (33,0)*+UR{ }="o3",    
    (36,0)*+UR{ }="o4",    
  (42,0)*+UR{\scriptstyle a}="a4",       
  "b1";"a1"**\dir{-}?>*\dir{>},      
  "a1";"o1"**\dir{-}?>*\dir{>},    
      "o1";"o2"**\dir{.}?>*\dir{.},                      
  "o2";"a2"**\dir{-}?>*\dir{>},   
    "a2";"a3"**\dir{-}?>*\dir{>},   
      "ai";"a4"**\dir{-}?>*\dir{>},   
    "a3";"o3"**\dir{-}?>*\dir{>},           
   "o3";"o4"**\dir{.}?>*\dir{.},    
   "o4";"a4"**\dir{-}?>*\dir{>},                                      
 (6,-7),{\ar@{.}@(l,l)(6,7)},  
  ?!{"b1";"b1"+/va(0)/}="b11", 
  ?!{"b1";"b1"+/va(30)/}="b12",
  ?!{"b1";"b1"+/va(-30)/}="b13",
  "b1";"b11"**\dir{-},  "b1";"b12"**\dir{-},  "b1";"b13"**\dir{-},
   (36,-7),{\ar@{.}@(r,r)(36,7)}, 
   ?!{"a4";"a4"+/va(0)/}="a41", 
  ?!{"a4";"a4"+/va(30)/}="a42",
  ?!{"a4";"a4"+/va(-30)/}="a43",
  "a4";"a41"**\dir{-},  "a4";"a42"**\dir{-},  "a4";"a43"**\dir{-},
    (6,7),{\ar@{.}(36,7)},
    (6,-7),{\ar@{.}(36,-7)},
\end{xy},
\end{align*}
working at the right end of the graph to finish our argument.  
Apply the Arnold identity to each of the heterogeneous edges originating at the final $a$ in the chain 
along with the homogeneous edge at the end, followed by antisymmetry to obtain graphs with arrows ``moving right.'' 
The first step is











\begin{align*}
 &G_n \,\,\,\, =
 \begin{xy}                           
  (12,5)*+UR{\scriptstyle a}="a3",    
  (9,0)*+UR{\scriptstyle a}="a1",   
  (15,0)*+UR{\scriptstyle a}="a2", 
    (18,5)*+UR{\scriptstyle c}="c1", 
    (3,0)*+UR{ }="o2",   
  "o2";"a1"**\dir{-}?>*\dir{>},         
  "a1";"a2"**\dir{-}?>*\dir{>},       
   "a2";"c1"**\dir{-}?>*\dir{>}, 
     "a3";"a2"**\dir{-}?>*\dir{>},               
    (9,-4),{\ar@{.}@(l,l)(9,8)}, 
  ?!{"o2";"o2"+/va(0)/}="o21", 
  "o2";"o21"**\dir{.}, 
   (15,-4),{\ar@{.}@(r,r)(15,8)}, 
   ?!{"a2";"a2"+/va(-30)/}="a21", 
    ?!{"a2";"a2"+/va(-50)/}="a22",
   ?!{"a2";"a2"+/va(-10)/}="a23", 
	   ?!{"c1";"c1"+/va(10)/}="c11",  
	   ?!{"c1";"c1"+/va(30)/}="c21",
	   ?!{"c1";"c1"+/va(50)/}="c31",
 "c1";"c11"**\dir{-},
 "c1";"c21"**\dir{-},
 "c1";"c31"**\dir{-},
      "a2";"a23"**\dir{-},  
      "a2";"a21"**\dir{-},  
   (9,8),{\ar@{.}(15,8)}, 
  "a2";"a22"**\dir{-}, 
    (9,-4),{\ar@{.}(15,-4)}, 
\end{xy}= \,\,\,\, -
\begin{xy}                           
  (12,5)*+UR{\scriptstyle a}="a3",    
  (9,0)*+UR{\scriptstyle a}="a1",   
  (15,0)*+UR{\scriptstyle a}="a2", 
    (18,5)*+UR{\scriptstyle c}="c1", 
    (3,0)*+UR{ }="o2",   
  "o2";"a1"**\dir{-}?>*\dir{>},         
  "a1";"a2"**\dir{-}?>*\dir{>},       
   "a2";"c1"**\dir{-}?>*\dir{>}, 
     "c1";"a3"**\dir{-}?>*\dir{>},               
    (9,-4),{\ar@{.}@(l,l)(9,8)}, 
  ?!{"o2";"o2"+/va(0)/}="o21", 
  "o2";"o21"**\dir{.}, 
   (15,-4),{\ar@{.}@(r,r)(15,8)}, 
   ?!{"a2";"a2"+/va(-30)/}="a21", 
    ?!{"a2";"a2"+/va(-50)/}="a22",
   ?!{"a2";"a2"+/va(-10)/}="a23", 
?!{"c1";"c1"+/va(10)/}="c11",  
	   ?!{"c1";"c1"+/va(30)/}="c21",
	   ?!{"c1";"c1"+/va(50)/}="c31",
 "c1";"c11"**\dir{-},
 "c1";"c21"**\dir{-},
 "c1";"c31"**\dir{-},
      "a2";"a23"**\dir{-},  
      "a2";"a21"**\dir{-},  
   (9,8),{\ar@{.}(15,8)}, 
  "a2";"a22"**\dir{-}, 
    (9,-4),{\ar@{.}(15,-4)}, 
\end{xy} -
\begin{xy}                           
  (12,5)*+UR{\scriptstyle a}="a3",    
  (9,0)*+UR{\scriptstyle a}="a1",   
  (15,0)*+UR{\scriptstyle a}="a2", 
    (18,5)*+UR{\scriptstyle c}="c1", 
    (3,0)*+UR{ }="o2",   
  "o2";"a1"**\dir{-}?>*\dir{>},         
  "a1";"a2"**\dir{-}?>*\dir{>},       
   "a3";"c1"**\dir{-}?>*\dir{>}, 
     "a2";"a3"**\dir{-}?>*\dir{>},               
    (9,-4),{\ar@{.}@(l,l)(9,8)}, 
  ?!{"o2";"o2"+/va(0)/}="o21", 
  "o2";"o21"**\dir{.}, 
   (15,-4),{\ar@{.}@(r,r)(15,8)}, 
   ?!{"a2";"a2"+/va(-30)/}="a21", 
    ?!{"a2";"a2"+/va(-50)/}="a22",
   ?!{"a2";"a2"+/va(-10)/}="a23", 
?!{"c1";"c1"+/va(10)/}="c11",  
	   ?!{"c1";"c1"+/va(30)/}="c21",
	   ?!{"c1";"c1"+/va(50)/}="c31",
 "c1";"c11"**\dir{-},
 "c1";"c21"**\dir{-},
 "c1";"c31"**\dir{-},
      "a2";"a23"**\dir{-},  
      "a2";"a21"**\dir{-},  
   (9,8),{\ar@{.}(15,8)}, 
  "a2";"a22"**\dir{-}, 
    (9,-4),{\ar@{.}(15,-4)}, 
\end{xy}.
\end{align*}

The first graph on the right hand side has fewer homogeneous edges. 
We then similarly apply the Arnold identity to the second graph on the right hand side,
using another normal heterogeneous edge which emanates from  final $a$ in the original linear chain (which is the next to final $a$ currently), 
followed again by antisymmetry.  Doing so 
for all of these normal edges, in each case a negative sign arising from the Arnold identity cancels with one from antisymmetry.
As the second graph on the right hand side occurs with a coefficient of $-1$ at the start of the process, we deduce that  $G_n \sim -G_0$.  Thus, we have
$G_1 = G_n-(n-1)G_0 \sim -nG_0$.  Since $G_0 \sim G_1$ and $G \sim G_0$ we deduce that $(n+1) G$ is equivalent to a linear combination
of graphs with fewer homogeneous edges, completing the reduction argument.  
The base case of no homogeneous edges is
a tautology.
\end{proof}

See Appendix~\ref{reductionexample} for an example of the reduction given by the proof of this theorem.



\section{The main theorem, and connections}

\subsection{Proof of the main theorem}

Recall Corollary~\ref{lcsdefined} that our letter-linking invariants are well-defined on  lower central series subquotients. 
Our main result in this paper is the following.   

\begin{theorem}\label{main}
The homomorphisms $\Phi : \Q \sy_i \to {\rm Hom}(\gamma_i F_n / \gamma_{i+1} F_n, \Q)$
are surjective.
\end{theorem}

Before proving this, we need a last calculational tool, motivated by the operadic approach to free Lie algebras.
We relate values of our letter-linking invariants under homomorphisms 
 induced by set maps of generators.
 
 \begin{definition}
 Let $F_S$ denote the free group on the generating set $S$.
 Let $f : S \to T$ be a map of generating sets, and let $f_*$ denote the induced homomorphism on free groups as well as the induced
 map on the set of symbols.
 Let $\Sigma_f$ denote the automorphisms of $S$ which commute with $f$.
 \end{definition}
 
The automorphisms $\Sigma_f$ are  isomorphic to a product of symmetric groups.

\begin{theorem}\label{lifts}
Let $f : S \to T$ be a map of (generating) sets, $w \in UC(n)$ and $\sigma$ a symbol on $S$ of depth $n$.  Then 
$$\Phi_{f_* \sigma} (f_* w) = \sum_{p \in \Sigma_f} \Phi_{p \cdot \sigma} (w).$$
\end{theorem}

We use this to understand letter-linking invariants with repeated letters, relating them to those with
unique letters, which are understood through Theorem~\ref{graph-symbol-pairing}.  

\begin{proof}[Proof of Theorem~\ref{lifts}]

We first show that 
$$\Phi_{f_* \sigma} (f_* w) = \sum_{ \widetilde{\sigma} |  f_*(\widetilde{\sigma}) = f(\sigma)} \Phi_{\widetilde{\sigma}} (w).$$
The sum contains the sum named in the theorem, along with additional terms which we will show vanish.

We prove this equality through analysis of lists, showing inductively that 
$\Lambda_{f_* \mu} (f_* w) \cong \bigcup  \Lambda_{\widetilde{\mu}} (w)$ -- that is, that these are in bijective correspondence respecting $f$ -- 
where the  union is over $\widetilde{\mu}$ such that $f_*(\widetilde{\mu}) = f_*(\mu)$.
For $\mu$ of depth zero, that is lists of occurrences of some generator, this is immediate.
Suppose this equality of lists holds for $\mu_1$ and $\mu_2$ of depth less than $n$.  By Theorem~\ref{definedonlcs} all of the $\Phi_{\widetilde{\mu_1}} (w)$
vanish, so we may choose all $d^{-1} \Lambda_{\widetilde{\mu_1}} (w)$.  Through our inductive bijection, the images of these under $f_*$ gives a choice of 
$d^{-1} \Lambda_{f_* \mu_1} (f_* w)$.  Moreover, each intersection of $\Lambda_{f_* \mu_2} (f_*w)$ with this cobounding corresponds to  the intersection of
some $ \Lambda_{\widetilde{\mu_2}} (w)$ with a $d^{-1} \Lambda_{\widetilde{\mu_1}} (w)$.  Through the bijective correspondence of
the product of the set of $\widetilde{\mu_1}$ over $\mu_1$
and the set of $\widetilde{\mu_2}$ over $\mu_2$ with the set of $\widetilde{(\mu_1) \mu_2}$ over $(\mu_1) \mu_2$, we establish
our inductive step that $\Lambda_{f_* (\mu_1) \mu_2 } (f_* w) \cong \bigcup  \Lambda_{\widetilde{(\mu_1) \mu_2}} (w)$ and thus
our first equality.

To deduce the equality of the theorem we see that  terms in the sum 
 $\sum_{ \widetilde{\sigma} |  f_*(\widetilde{\sigma}) = f(\sigma)} \Phi_{\widetilde{\sigma}} (w)$  which are not of the 
form $p \cdot \sigma$ for $p \in \Sigma_f$ must have some repeated letter.  But  $w \in UC(n)$, so there will be at least one  letter 
which occurs in $w$ but not $\widetilde{\sigma}$.  That letter can then be removed from $w$ without changing $\Lambda_{\widetilde{\sigma}}$.  
But removing the letter from $w$ is equivalent to replacing the letter by the identity element.  Since $w$ is a commutator the resulting word would
represent the identity element.
\end{proof}

We now extend Theorem~\ref{graph-symbol-pairing} from graphs with unique vertex labels to all distinct-vertex graphs.

\begin{definition}
Let $w \in \gamma_i F_n$, and let $W$ be the span of generators of $F_n$. We set the {\bf Lie image} of $w$,
denoted  $\lambda(w) \in \fL_n \cong \Lie(i) \otimes W$, 
to be the image of $w$ in $ \gamma_i F_n / \gamma_{i+1} F_n$, composed with its isomorphism with the $i$th graded component
of the free Lie algebra.
\end{definition}

In particular, $\lambda(w)$ converts a commutator in the generators (which we call a {\bf basic commutator}) 
to the corresponding Lie bracket.  Theorem~\ref{lifts} leads to the following.

\begin{corollary}\label{last}
$\Phi_G (w) = \langle G, \lambda(w) \rangle$, where $\langle -,  -\rangle$ denotes the configuration pairing.
\end{corollary}

\begin{proof}
By linearity, it suffices to consider basic commutators.
Let $w$ be a basic commutator and $\tilde{w}$ be a basic commutator in unique letters 
so that $f_*(\tilde{w}) = w$ for some map of generating sets $S$.  For $\Phi_G (w)$ to be non-zero
there must be a $\widetilde{G}$ with $f_*(\widetilde{G}) = G$.  Theorem~\ref{lifts} then gives a formula for $\Phi_G(w)$.  But
we can apply Theorem~\ref{graph-symbol-pairing} to every term in the right-hand side.  Doing so we obtain terms in the definition 
of $\langle G, \lambda(w) \rangle$, which is the extension of the pairing between $\Eil(i)$ and $\Lie(i)$ and Kronecker pairing
on $W$ to 
$\Eil(i) \otimes_{\si_n} W^{\otimes i}$ and $\Lie(i) \otimes_{\si_n} W^{\otimes i}$.
The terms in this extension which do not occur in the application of Theorem~\ref{lifts} will not contribute to this sum, as the Kronecker pairing will be zero.
\end{proof}

The proof of our main result is now a matter of assembly.

\begin{proof}[Proof of Theorem~\ref{main}]
 By Proposition~\ref{reductionsurjective}, any symbol $\sigma \in \sy_n = \gs_{n,0}$ 
 is the reduction of some graph $G$ in $\gs_{0,n}$. 
 By Corollary~\ref{last}, the values of $\Phi_\sigma$ on $\gamma_i F_n$ coincide with the configuration pairing
 of $G$ on the $i$-graded summand of $\fL_n$.
 By Theorem~\ref{distinctspan}, configuration pairings with distinct-vertex graphs span the functionals given by all graphs.   
 By Corollary~3.3 of \cite{SiWa11} pairing with all such graphs modulo Arnold and antisymmetry is perfect
 on this $i$-graded summand, which is isomorphic to $\gamma_i F_n / \gamma_{i+1} F_n$.
\end{proof}

\subsection{Comparison with Fox derivatives}

There is already a well-known collection of homomorphisms which span the linear dual of the lower-central series Lie algebra for 
free groups, namely those
given by  Fox's free differential calculus \cite{Fox53, CFL58}, whose definition we recall below.    These differ from the functionals we 
provide in substantial ways.

\begin{itemize}
\item Fox derivatives span homomorphisms to the integers, while letter-linking homomorphisms only span over the rationals.  See  Appendix~\ref{values}.
\item Fox derivatives are defined on the entire free group, while letter-linking homomorphisms are only defined on subgroups.
\item As shown below, Fox derivatives correspond to evaluation of 
the linear graph spanning set for the cofree Lie coalgebra $\E_n$, while letter-linking homomorphisms correspond to evaluation of the
distinct-vertex spanning set.  
\item Fox derivatives, as developed in part by Chen, Fox and Lyndon \cite{CFL58}, 
are more immediately compatible with the Chen model of rational homotopy theory while 
letter-linking homomorphisms are drawn from the Quillen model.
\item For hand calculations, letter-linking numbers involve fewer calculations, though we conjecture below that Fox derivatives could be modified
to involve similar calculation. 
\item For fundamental groups of punctured surfaces, letter-linking invariants immediately
give rise to lower bounds on the complexity of curves which represent elements
of $\gamma_i F_n$.
\end{itemize}

With our eyes towards applications to mapping class groups (first author) and non-simply connected
rational homotopy theory (second author) we believe the first two 
properties in which letter-linking homomorphisms are inferior are worth trading for the properties in which 
they are superior.

We now make the  connection between Fox derivatives and our model for cofree Lie coalgebras, starting with the definition
of the former.

\begin{definition}
Let $F_n$ be the free group on $n$ generators and let $\alpha: \Z[F_n]\rightarrow \Z$ be the augmentation. 
A derivation $D$ is a map $D: \Z[F_n] \rightarrow \Z[F_n]$  such that 
\begin{enumerate}
\item $D(u+v)=Du+Dv $
\item $D(uv)=Du\cdot \alpha(v)+ u \cdot Dv$
\end{enumerate}
\end{definition}

\begin{theorem}\cite{Fox53}
Let $x_1,\ldots, x_n$ denote the generators of the free group $F_n$.  There is a unique derivation
\[
\frac{\partial}{\partial x_i} : \Z [F_n] \rightarrow \Z[F_n]
\]
such that $\frac{\partial}{\partial x_i}(x_j)=\delta_{i,j}$, 
the Kronecker delta.  We call this the Fox derivative with respect to $x_i$ and denote it by $\partial_{x_i}$.  
\end{theorem}

This derivation is then iterated. 

\begin{definition}
Let $c = a_1, \ldots, a_k$ be some collection of the generators $x_1,\ldots, x_n$, with repeats  allowed.  
 For $v \in \Z [F_n]$ inductively define
\[
\partial_{a_1,\cdots, a_k}(v)=\partial_{a_1}(\partial_{a_2,\cdots ,a_k}(v)).
\]
Define $\partial_c^\circ(v)$ to be $\alpha(\partial_c(v))$.

\end{definition}

In Applendix~\ref{examplecalculation} we give an example of a Fox derivative calculation.  Informally, a 
derivative takes every monomial and produces a sum of monomials by ``cutting'' it at each occurrence of a generator, with signs.
For example, $\partial_{a,b}$ will cut along occurrences of $b$ then along occurences of 
$a$, and then through augmentation count the resulting monomials with signs.
Effectively, this counts occurrences of an $a$ followed by a $b$.  But for any $a$ followed by  an $a^{-1}$, any subsequent $b$'s
will be counted with both a $+1$ for the  $a$ and a $-1$ for the $a^{-1}$.  Thus we could streamline the calculation by counting only $b$'s between
$a$-$a^{-1}$ pairs -- that is, the count $\Phi_{(a)b}$.

After depth three, Fox derivatives and letter-linking invariants do not coincide, as follows from calculations in  
Appendix~\ref{values} and its generalizations.
But we could express Fox derivatives in the same graphical context we use in Section~\ref{examples}  to calculate and
understand our letter-linking invariants.

In \cite{CFL58} the authors produce a collection of $c=a_1\cdots a_k$ so that $\partial_c^\circ$ 
form a basis for the dual space of each $\gamma_i F_n/ \gamma_{i+1}F_n$.  We produce a new proof of this fact in
order to compare  Fox derivatives with letter-linking invariants, starting with the analogue of Corollary~\ref{last}.

\begin{theorem}
Let $c=a_1, \cdots ,a_k$ and let $G_c$ be the graph $\longgraph{a_1}{a_{2}}{\cdots}{a_{k}}$, and let $w \in \gamma_k F_n$.   
Then $\partial^\circ_c w = \langle G_c, \lambda (w) \rangle$.
\end{theorem}

\begin{proof}
Equation (3.3) of \cite{CFL58} states that for $u \in \gamma_i F_n$, $v \in \gamma_j F_n$ with $i + j = k$,
$$\partial^\circ_c([u,v])= \partial^\circ_{c_f} (u) \partial^\circ_{c_l} (v)  - \partial^\circ_{c'_l} (u)  \partial^\circ_{c'_f} (v),$$
where $c_f = a_1, \cdots, a_i$, $c'_f = a_1, \cdots, a_j$, and $c_l$ and $c'_l$ are their complements in $c$.

We compare this equality  with the bracket-cobracket formula  
$\langle G_c, \lambda ([u,v]) \rangle = ] G_c [ u \otimes v$, established in Corollary 3.14 of \cite{SiWa11}.  
Since $] G_c [ = G_{c_f} \otimes G_{c_l} - G_{c'_l} \otimes G_{c'_f}$,  these formulae agree.
Because these formulae determine the values of the Fox derivatives and graph coalgebra pairings, they establish
the theorem inductively, starting with the weight zero case which is immediate.
\end{proof}

As developed in \cite{CFL58}, 
the bracket-cobracket reduction formula for Fox derivatives then shows that they yield the coefficients of 
the map from the free Lie algebra to its universal enveloping algebra, which is the tensor algebra.  
By Remark~1.5 of \cite{Walt20}, the resulting functionals on the free Lie algebra
are represented by the linear graph spanning set, or basis if one 
choses a subset of linear graphs such as 
Lyndon-Shirsov words.  But  such linear graphs are not generally distinct-vertex graphs.
We show in Appendix~\ref{values} that the spanning set for linear functions we develop, 
represented by distinct vertex graphs, is distinct from this classical spanning set.

Both letter-linking invariants and Fox derivatives are roughly order $n^d$ to compute, where $n$ is the length of the word and $d$ is the depth,
as both can be viewed as producing and counting with signs on the order of $n^d$ sub-words.
Based on the argument that $\partial_{a,b} = \Phi_{(a)b}$ above and the examples in Appendix~\ref{examplecalculation}, we
conjecture that with finer analysis the ``from the definition'' calculation of letter-linking invariants is more efficient 
than that of Fox derivatives, but that they become comparable to compute once
cancellation in Fox derivative expansion is systematically accounted for.  We suggest such analysis for further work, perhaps after
all of these techniques can be extended  to other groups, as suggested in the next section. 

\subsection{Further directions}

We expect our new insight into the lower central series Lie algebra of free groups, first 
studied by Magnus eighty years ago \cite{Magn37}, will have impact in a few directions. 

In algebra, an immediate question is whether and how letter-linking invariants could be generalized to arbitrary finitely presented groups. 
For example,  the fundamental group of the genus-two surface has four generators $a, \ldots d$ and the relation $[a,b]=[c,d]$.  We conjecture
there is a complete collection of letter-linking invariants which now include the linear combination $\Phi_{(a)b} + \Phi_{(c)d}$, but neither count on
its own.   We can see this invariant in the context of the formalism developed in this paper and in \cite{SiWa11} as follows.
The Lie coalgebraic bar complex on the cochains of a space provides the setting for Hopf invariants in higher dimensions \cite{SiWa13}.
In this paper the space in question has been a wedge of circles, whose cochains are equivalent to the first cohomology (that is, this space
is formal), resulting
in the bar complex being equivalent to the cofree Lie coalgebra $\E_n$.    In this setting of a surface, we still have formality, with the cohomology
generated by classes $A,B,C,D$, say Kronecker dual to the homology classes of $a,\ldots,d$, with the relation $AB = -CD$.  In the bar complex,
$\linep{A}{B}$ will not be a cycle, having coboundary $AB$, but $\linep{A}{B} + \linep{C}{D}$ will be a cycle.  It  will reduce to 
the proposed invariant $\Phi_{(a)b} + \Phi_{(c)d}$, whose well-definedness seems more delicate, restricted to the first commutator subgroup rather than the 
domain of definitions of these counts.
We expect the Lie coalgebraic bar complex to control these letter-linking invariants in general.
To our knowledge, Fox derivative techniques
have not been extended to the lower central series of other groups, so such an extension would break substantial new ground.

If  such bar complexes  produce the dual to the lower central series Lie algebra of a group, they could likely be merged with
the Lie coalgebraic models for rational homotopy in the simply connected setting.  A primary issue to resolve is that the notion of distinct vertices, which is
essential to defining letter-linking invariants, does not have a natural counterpart in higher dimensions.   
If such models can be developed, they could then be compared with 
new Lie models of Buijs-F\'elix-Murillo-Tanr\'e  \cite{BFMT18}, which are based on the Lawrence-Sullivan cylinder object \cite{LaSu14}.  These
new models are promising but have yielded relatively few calculations. 

There are plenty of elementary questions as well.  Even in the free group case, a finer comparison of Fox derivatives as counting
occurrences of sequences of letters and letter-linking invariants would be interesting.  While we know that distinct-vertex
graphs span cofree Lie coalgebras on a set of (co)generators, we have yet to find a basis.  It would be interesting to connect
such a basis, as well as closer
analysis of relations, to the literature on (distinctly) colored trees.  Looking at the examples
in Section~\ref{values}, it seems likely that understanding the values of that basis on free Lie algebras could lead to new bases for the latter.
These examples also point to the question of computing the indices of the functionals arising from letter-linking invariants 
within all integer-valued functionals.  One should decompose the free Lie algebra on a generating set 
into summands by the number of times each generator occur and compute on those summands, in which case so far we only see factorials arise. 



\appendix
\section{Further examples}
\subsection{Values on a Free Lie algebra basis}\label{values}

We choose a basis for our letter-linking invariants and share its values on a choice of Hall basis, 
which is also the Lyndon basis,
for $\gamma_5 F_2 / \gamma_6 F_2$.
This pairing decomposes into blocks, according to the number of times each generator, which we call $a$ and $b$, occur.

At the extremes, we have the  $[a,[a,[a,[a,b]]]]$, the only basis element with four $a$'s. 
Here there is a unique  linking invariant  symbol,
$(a)(a)(a)(a)b$, 
which is reduction of the distinct-vertex graph.
\begin{xy}  
  (0,0)*+UR{\scriptstyle a}="a1",    
  (6,0)*+UR{\scriptstyle b}="b1",   
  (12,0)*+UR{\scriptstyle a}="a2", 
  (6,6)*+UR{\scriptstyle a}="a3", 
   (6,-6)*+UR{\scriptstyle a}="a4", 
   "a1";"b1"**\dir{-}?>*\dir{>},
   "a2";"b1"**\dir{-}?>*\dir{>},
   "a3";"b1"**\dir{-}?>*\dir{>},
   "a4";"b1"**\dir{-}?>*\dir{>},  
 \end{xy}.  The value of the invariant is $24$.  
 The case of only one $a$ and four $b$'s is similar.

With three $a$'s and two $b's$ there are Hall basis elements $[a,[a,[[a,b],b]]]$ and $[[a,[a,b]],[a,b]]$.  Letter linking symbols are
$(((b)a)(a) b)a$ and $((((a)b)a)b)a$, the former being the reduction of \begin{xy}  
  (0,0)*+UR{\scriptstyle b}="b1",    
  (6,0)*+UR{\scriptstyle a}="a1",   
  (12,0)*+UR{\scriptstyle b}="b2", 
  (18,0)*+UR{\scriptstyle a}="a2", 
   (9,5)*+UR{\scriptstyle a}="a3", 
   "b1";"a1"**\dir{-}?>*\dir{>},
   "a1";"b2"**\dir{-}?>*\dir{>},
   "b2";"a2"**\dir{-}?>*\dir{>},
   "a3";"b2"**\dir{-}?>*\dir{>},  
 \end{xy} and the latter being the reduction of the ``linear'' graph \begin{xy}  
  (0,0)*+UR{\scriptstyle a}="a1",    
  (3,5)*+UR{\scriptstyle b}="b1",   
  (6,0)*+UR{\scriptstyle a}="a2", 
  (9,5)*+UR{\scriptstyle b}="b2", 
   (12,0)*+UR{\scriptstyle a}="a3", 
   "a1";"b1"**\dir{-}?>*\dir{>},
   "b1";"a2"**\dir{-}?>*\dir{>},
   "a2";"b2"**\dir{-}?>*\dir{>},
   "b2";"a3"**\dir{-}?>*\dir{>},  
 \end{xy} .
The pairing here is not a Kronecker pairing, being represented by the matrix $\begin{bmatrix}  4 & -2 \\ 4 & 4 \end{bmatrix}$.

Next, with two $a$'s and three $b$'s we have Hall basis elements $[a,[[[a,b],b,],b]]$ and $[[a,b],[[a,b],b]]$,  
and letter-linking symbols $(((a)b) (b)a)b$ and $((((b)a)b)a)b$, 
the former being the reduction of \begin{xy}  
  (0,0)*+UR{\scriptstyle a}="a1",    
  (6,0)*+UR{\scriptstyle b}="b1",   
  (12,0)*+UR{\scriptstyle a}="a2", 
  (18,0)*+UR{\scriptstyle b}="b2", 
   (9,5)*+UR{\scriptstyle b}="b3", 
   "a1";"b1"**\dir{-}?>*\dir{>},
   "b1";"a2"**\dir{-}?>*\dir{>},
   "a2";"b2"**\dir{-}?>*\dir{>},
   "b3";"a2"**\dir{-}?>*\dir{>},  
 \end{xy} and the latter being the reduction of the ``linear'' graph \begin{xy}  
  (0,0)*+UR{\scriptstyle b}="b1",    
  (3,5)*+UR{\scriptstyle a}="a1",   
  (6,0)*+UR{\scriptstyle b}="b2", 
  (9,5)*+UR{\scriptstyle a}="a2", 
   (12,0)*+UR{\scriptstyle b}="b3", 
   "b1";"a1"**\dir{-}?>*\dir{>},
   "a1";"b2"**\dir{-}?>*\dir{>},
   "b2";"a2"**\dir{-}?>*\dir{>},
   "a2";"b3"**\dir{-}?>*\dir{>},  
 \end{xy} .  Here the pairing represented by the matrix $\begin{bmatrix}  6 & -2 \\ 0 & 4 \end{bmatrix}$,
yielding the same determinant (index) as in the previous case.

Other choices for representative 
letter-linking invariants give the same results, up to sign.  Thus the Hall basis, which in this case is also the Lyndon basis, is not Kronecker 
in pairing with letter-linking invariants.    It would be interesting to see such a dual basis in general, since it seems like it would
have symmetry properties.  Bases which use orderings on the generating
set, both classical bases as well as new ones such as those in \cite{ShWa16},  do not have such symmetry, so this would be a new tool
in the study of free Lie algebras.

\subsection{Reduction to distinct vertex Eil graphs}  \label{reductionexample}

We reduce the graph
\begin{align*}
G=
\begin{xy}  
  (0,0)*+UR{\scriptstyle b}="b1",    
  (6,0)*+UR{\scriptstyle a}="a1",   
  (12,0)*+UR{\scriptstyle a}="a2", 
  (15,5)*+UR{\scriptstyle c}="c1", 
   (15,-5)*+UR{\scriptstyle d}="d1", 
   "b1";"a1"**\dir{-}?>*\dir{>},
   "a1";"a2"**\dir{-}?>*\dir{>},
   "a2";"c1"**\dir{-}?>*\dir{>},
   "a2";"d1"**\dir{-}?>*\dir{>},  
 \end{xy}  
\end{align*}
to a rational linear combination of distinct vertex graphs, following the procedure outlined in the proof of 
Theorem~\ref{distinctspan}.

 There is only one maximal subgraph of $a$'s, and it is already linear, so $G=G_0$. Thus, the first step is applying the Arnold  identity to the first two edges to get
\begin{align*}
G_0= \ - \
\begin{xy}                          
  (0,0)*+UR{\scriptstyle b}="b1",    
  (3,5)*+UR{\scriptstyle a}="a1",   
  (6,0)*+UR{\scriptstyle a}="a2", 
  (12,0)*+UR{\scriptstyle c}="c1", 
   (9,-5)*+UR{\scriptstyle d}="d1", 
   "a2";"b1"**\dir{-}?>*\dir{>},
   "a1";"a2"**\dir{-}?>*\dir{>},
   "a2";"c1"**\dir{-}?>*\dir{>},
   "a2";"d1"**\dir{-}?>*\dir{>},
 \end{xy} \ - \ 
 \begin{xy}                           
  (6,0)*+UR{\scriptstyle b}="b1",    
  (0,0)*+UR{\scriptstyle a}="a1",   
  (12,0)*+UR{\scriptstyle a}="a2", 
  (15,5)*+UR{\scriptstyle c}="c1", 
   (15,-5)*+UR{\scriptstyle d}="d1", 
   "a2";"b1"**\dir{-}?>*\dir{>},
   "b1";"a1"**\dir{-}?>*\dir{>},
   "a2";"c1"**\dir{-}?>*\dir{>},
   "a2";"d1"**\dir{-}?>*\dir{>}, 
 \end{xy},
 \end{align*}
which applying antisymmetry implies
\begin{align*} 
G_0=
\begin{xy}                           
  (0,0)*+UR{\scriptstyle b}="b1",    
  (3,5)*+UR{\scriptstyle a}="a1",   
  (6,0)*+UR{\scriptstyle a}="a2", 
  (12,0)*+UR{\scriptstyle c}="c1", 
   (9,-5)*+UR{\scriptstyle d}="d1", 
   "b1";"a2"**\dir{-}?>*\dir{>},
   "a1";"a2"**\dir{-}?>*\dir{>},
   "a2";"c1"**\dir{-}?>*\dir{>},
   "a2";"d1"**\dir{-}?>*\dir{>},
 \end{xy} \ - \ \begin{xy}                           
  (0,0)*+UR{\scriptstyle a}="a1",    
  (6,0)*+UR{\scriptstyle b}="b1",   
  (12,0)*+UR{\scriptstyle a}="a2", 
  (15,5)*+UR{\scriptstyle c}="c1", 
   (15,-5)*+UR{\scriptstyle d}="d1", 
   "a1";"b1"**\dir{-}?>*\dir{>},
   "b1";"a2"**\dir{-}?>*\dir{>},
   "c1";"a2"**\dir{-}?>*\dir{>},
   "d1";"a2"**\dir{-}?>*\dir{>},
\end{xy}.
\end{align*}


In the notation of Theorem~\ref{distinctspan},  the first graph on the right hand side above is  $G_1$, which in this case is also our $G_n$.
The second graph, which has fewer homogeneous edges, was not given a name in the proof of the theorem, but for convenience  it will be called $H$. So $G_0=G_1-H$, which implies $G_0\sim G_1$. 

We manipulate our $G_1$ following the process given for
 $G_n$ in the proof of Theorem~\ref{distinctspan}.  In the first graph we apply the Arnold relation at the right end, which after redrawing gives
\begin{align*}
  G_1=\ - \
\begin{xy}                           
  (0,0)*+UR{\scriptstyle b}="b1",    
  (6,0)*+UR{\scriptstyle a}="a1",   
  (18,0)*+UR{\scriptstyle a}="a2", 
  (12,0)*+UR{\scriptstyle c}="c1", 
   (9,-5)*+UR{\scriptstyle d}="d1", 
   "b1";"a1"**\dir{-}?>*\dir{>},
   "a1";"c1"**\dir{-}?>*\dir{>},
   "c1";"a2"**\dir{-}?>*\dir{>},
   "a1";"d1"**\dir{-}?>*\dir{>},
 \end{xy}\ - \
 \begin{xy}                           
  (0,0)*+UR{\scriptstyle b}="b1",    
  (6,0)*+UR{\scriptstyle a}="a1",   
  (12,0)*+UR{\scriptstyle a}="a2", 
  (18,0)*+UR{\scriptstyle c}="c1", 
   (9,-5)*+UR{\scriptstyle d}="d1", 
   "b1";"a1"**\dir{-}?>*\dir{>},
   "a2";"a1"**\dir{-}?>*\dir{>},
   "c1";"a2"**\dir{-}?>*\dir{>},
   "a1";"d1"**\dir{-}?>*\dir{>},
 \end{xy}.
\end{align*}

The first graph is a distinct vertex graph with fewer homogeneous edges. 
We apply the Arnold identity to the second graph to get

\begin{align*}
 G_1 = \ - \
\begin{xy}                           
  (0,0)*+UR{\scriptstyle b}="b1",    
  (6,0)*+UR{\scriptstyle a}="a1",   
  (18,0)*+UR{\scriptstyle a}="a2", 
  (12,0)*+UR{\scriptstyle c}="c1", 
   (9,-5)*+UR{\scriptstyle d}="d1", 
   "b1";"a1"**\dir{-}?>*\dir{>},
   "a1";"c1"**\dir{-}?>*\dir{>},
   "c1";"a2"**\dir{-}?>*\dir{>},
   "a1";"d1"**\dir{-}?>*\dir{>},
 \end{xy}\ + \
 \begin{xy}                           
  (0,0)*+UR{\scriptstyle b}="b1",    
  (6,0)*+UR{\scriptstyle a}="a1",   
  (18,0)*+UR{\scriptstyle a}="a2", 
  (24,0)*+UR{\scriptstyle c}="c1", 
   (12,0)*+UR{\scriptstyle d}="d1", 
   "b1";"a1"**\dir{-}?>*\dir{>},
   "d1";"a2"**\dir{-}?>*\dir{>},
   "c1";"a2"**\dir{-}?>*\dir{>},
   "a1";"d1"**\dir{-}?>*\dir{>},
 \end{xy}\ + \
 \begin{xy}                           
  (0,0)*+UR{\scriptstyle b}="b1",    
  (6,0)*+UR{\scriptstyle a}="a1",   
  (12,0)*+UR{\scriptstyle a}="a2", 
  (15,5)*+UR{\scriptstyle c}="c1", 
   (15,-5)*+UR{\scriptstyle d}="d1", 
   "b1";"a1"**\dir{-}?>*\dir{>},
   "a2";"a1"**\dir{-}?>*\dir{>},
   "c1";"a2"**\dir{-}?>*\dir{>},
   "d1";"a2"**\dir{-}?>*\dir{>},
   \end{xy}.
\end{align*}

Rewriting using the antisymmetry relation,

\begin{align*}
G_1  = 
\begin{xy}                           
  (0,0)*+UR{\scriptstyle b}="b1",    
  (6,0)*+UR{\scriptstyle a}="a1",   
  (18,0)*+UR{\scriptstyle a}="a2", 
  (12,0)*+UR{\scriptstyle c}="c1", 
   (9,-5)*+UR{\scriptstyle d}="d1", 
   "b1";"a1"**\dir{-}?>*\dir{>},
   "a1";"c1"**\dir{-}?>*\dir{>},
   "c1";"a2"**\dir{-}?>*\dir{>},
   "d1";"a1"**\dir{-}?>*\dir{>},
 \end{xy}\ -\
 \begin{xy}                           
  (0,0)*+UR{\scriptstyle b}="b1",    
  (6,0)*+UR{\scriptstyle a}="a1",   
  (18,0)*+UR{\scriptstyle a}="a2", 
  (24,0)*+UR{\scriptstyle c}="c1", 
   (12,0)*+UR{\scriptstyle d}="d1", 
   "b1";"a1"**\dir{-}?>*\dir{>},
   "d1";"a2"**\dir{-}?>*\dir{>},
   "a2";"c1"**\dir{-}?>*\dir{>},
   "a1";"d1"**\dir{-}?>*\dir{>},
 \end{xy}\ -\
 \begin{xy}                           
  (0,0)*+UR{\scriptstyle b}="b1",    
  (6,0)*+UR{\scriptstyle a}="a1",   
  (12,0)*+UR{\scriptstyle a}="a2", 
  (15,5)*+UR{\scriptstyle c}="c1", 
   (15,-5)*+UR{\scriptstyle d}="d1", 
   "b1";"a1"**\dir{-}?>*\dir{>},
   "a1";"a2"**\dir{-}?>*\dir{>},
   "a2";"c1"**\dir{-}?>*\dir{>},
   "a2";"d1"**\dir{-}?>*\dir{>},
\end{xy}.   
\end{align*}

Notice, the rightmost graph is our original $G = G_0$, so $G_1\sim -G_0$.  As we first showed above that $G=G_1-H$, 
we substitute $G_1=G+H$ into the left hand side of the equation above and solve for $G$ to get,

\begin{align*}
G  =  \frac{1}{2}
\begin{xy}
(-3,0)*{\text{\Huge(} },
  (0,0)*+UR{\scriptstyle b}="b1",    
  (6,0)*+UR{\scriptstyle a}="a1",   
  (18,0)*+UR{\scriptstyle a}="a2", 
  (12,0)*+UR{\scriptstyle c}="c1", 
   (9,-5)*+UR{\scriptstyle d}="d1", 
   "b1";"a1"**\dir{-}?>*\dir{>},
   "a1";"c1"**\dir{-}?>*\dir{>},
   "c1";"a2"**\dir{-}?>*\dir{>},
   "d1";"a1"**\dir{-}?>*\dir{>},
 \end{xy}\ - \
 \begin{xy}                           
  (0,0)*+UR{\scriptstyle b}="b1",    
  (6,0)*+UR{\scriptstyle a}="a1",   
  (18,0)*+UR{\scriptstyle a}="a2", 
  (24,0)*+UR{\scriptstyle c}="c1", 
   (12,0)*+UR{\scriptstyle d}="d1", 
   "b1";"a1"**\dir{-}?>*\dir{>},
   "d1";"a2"**\dir{-}?>*\dir{>},
   "a2";"c1"**\dir{-}?>*\dir{>},
   "a1";"d1"**\dir{-}?>*\dir{>},
 \end{xy}\ - \
 \begin{xy}                           
  (0,0)*+UR{\scriptstyle a}="a1",    
  (6,0)*+UR{\scriptstyle b}="b1",   
  (12,0)*+UR{\scriptstyle a}="a2", 
  (15,5)*+UR{\scriptstyle c}="c1", 
   (15,-5)*+UR{\scriptstyle d}="d1", 
   (18,0)*{\text{\Huge)}},
   "a1";"b1"**\dir{-}?>*\dir{>},
   "b1";"a2"**\dir{-}?>*\dir{>},
   "c1";"a2"**\dir{-}?>*\dir{>},
   "d1";"a2"**\dir{-}?>*\dir{>},
\end{xy}.
\end{align*}

\subsection{Letter-linking invariants and Fox derivatives}\label{examplecalculation}

We show how Fox derivatives work for the example given in Section 2.1. 
For comparison, the corresponding Fox derivative is $(\partial_a \partial_b \partial_a)^\circ(w)$.  One must calculate each derivative in turn.
First, 
\[
\partial_a(w)=1+a+aab-aabacb^{-1}c^{-1}a^{-1} -aabacb^{-1}c^{-1}a^{-1}a^{-1}-aabacb^{-1}c^{-1}a^{-1}a^{-1}cbc^{-1}a^{-1}.
\]
 Apply $\partial_b$ to each word to get
\begin{align*}
\partial_b(1)&=0,\\
\partial_b(a)&=0,\\
\partial_b(aab)&=aa,\\
\partial_b(aabacb^{-1}c^{-1}a^{-1})&=aa-aabacb^{-1},\\
\partial_b(aabacb^{-1}c^{-1}a^{-1}a^{-1})&=aa-aabacb^{-1},\\
\partial_b(aabacb^{-1}c^{-1}a^{-1}a^{-1}cbc^{-1}a^{-1})&=aa-aabacb^{-1}+aabacb^{-1}c^{-1}a^{-1}a^{-1}c.\\
\end{align*}
Substituting and simplifying, we have $\partial_b \partial_a(w)=-2aa+3aabacb^{-1}-aabacb^{-1}c^{-1}a^{-1}a^{-1}c$.  Compute $\partial_a$ of each term to get
\begin{align*}
\partial_a(2aa)&=2+2a,\\
\partial_a(3aabacb^{-1})&=3+3a+3aab,\\
\partial_a(aabacb^{-1}c^{-1}a^{-1}a^{-1})&=1+a+aab-aabacb^{-1}c^{-1}a^{-1}-aabacb^{-1}c^{-1}a^{-1}a^{-1}.
\end{align*}

Substitute the terms back into $\partial_b\partial_a(w)$, we have
\[
\partial_a\partial_b\partial_a(w)=2aab+aabacb^{-1}c^{-1}a^{-1}+aabacb^{-1}c^{-1}a^{-1}a^{-1}
\]
Apply the augmentation to obtain $(\partial_a\partial_b\partial_a)^\circ(w)=4$.  

This value agrees with our linking invariant, as it must
since their values both correspond to the functional on the free Lie algebra given by the graph \begin{xy}  
  (0,0)*+UR{\scriptstyle a}="b1",    
  (3,5)*+UR{\scriptstyle b}="a1",   
  (6,0)*+UR{\scriptstyle a}="b2", 
   "b1";"a1"**\dir{-}?>*\dir{>},
   "a1";"b2"**\dir{-}?>*\dir{>},
 \end{xy}.

Accounting for the algebra which was omitted, 
 the diagrammatic method  for letter-linking invariants leads to substantially easier work by hand than Fox derivatives.  In
our experience this difference increases as the
depth increases.  In further work,
the first author has found an algorithm which makes letter-linking invariants  more efficient than Fox derivatives computationally.



\bibliographystyle{alpha}
\bibliography{main}

\end{document}